\documentclass{article}
\usepackage{t1enc,amsfonts,latexsym,amsmath,amsthm,verbatim,amssymb,graphics,pstool,color,wrapfig,xcolor}

\DeclareMathOperator*{\argmin}{arg\,min }
\DeclareMathOperator*{\argmax}{arg\,max }
\newcommand{\fro}[1]{\left\| #1\right\|_2^2}
\newcommand{\scal}[1]{\left \langle #1 \right \rangle}
\newcommand{\para}[1]{\left( #1\right)}

\newtheorem{theorem}{Theorem}[section]
\newtheorem{proposition}[theorem]{Proposition}
\newtheorem{lemma}[theorem]{Lemma}
\newtheorem{corollary}[theorem]{Corollary}

\def\re{{\mathsf {Re}}}

\def\D{{\mathbb D}}

\def\C{{\mathbb C}}
\def\K{{\mathcal K}}

\def\G{{\mathcal G}}
\def\Q{{\mathcal Q}}

\def\R{{\mathbb R}}
\def\S{{\mathcal{S}}}
\def\kfi{{\mathcal{K}_{\mu}}}
\def\kfir{{\mathcal{K}_{\mu,reg}}}
\def\ksi{{\mathcal{K}_k}}
\def\ksir{{\mathcal{K}_{k,reg}}}

\def\ep{{{ \epsilon}}}

\def\card{{\text{card}}}
\def\supp{{\text{supp }}}

\newcommand{\skal}[1]{\langle #1 \rangle}

\title{An unbiased approach to compressed sensing}
\author{Marcus Carlsson\thanks{Centre for Mathematical Sciences , Lund University,  {mc,gerosa,calle@maths.lth.se}}
\and Daniele Gerosa\footnotemark[1]
\and
Carl Olsson\footnotemark[1]
\thanks{Electrical Engineering, Chalmers University of Technology, caols@chalmers.se}
}
\date{}

\begin{document}
\maketitle
\begin{abstract}
In compressed sensing a sparse vector is approximately retrieved from an under-determined equation system $Ax=b$. Exact retrieval would mean solving a large combinatorial problem which is well known to be NP-hard. For $b$ of the form $Ax_0+\ep$, where $x_0$ is the ground truth and $\ep$ is noise, the ``oracle solution'' is the one you get if you a priori know the support of $x_0$, and is the best solution one could hope for. We provide a non-convex functional whose global minimum is the oracle solution, with the property that any other local minimizer necessarily has high cardinality. We provide estimates of the type
$\|\hat x-x_0\|_2\leq C\|\ep\|_2$ with constants $C$ that are significantly lower than for competing methods or theorems, and our theory relies on soft assumptions on the matrix $A$, in comparison with standard results in the field.

%If $b$ is of the form $Ax_0+\ep$ where $\ep$ is noise and $x_0$ is sparse, a celebrated result states that, under certain assumptions on $A$, the compressed sensing solution $\hat x$ satisfies \begin{equation}\label{2}\|\hat x-x_0\|_2\leq C_K\|\ep\|_2,\end{equation} where $C_K$ is a constant depending on $A$. We point out a number of weaknesses with this theory and show how these can be overcome at the price of using a non-convex framework. We relax the assumptions on $A$, provide an estimate corresponding to \eqref{2} with a much lower constant, and also demonstrate that the corresponding method works better in practice for the case of a Gaussian random $A$ with normalized columns.

The framework also allows to incorporate a priori information on the cardinality of the sought vector. In this case we show that despite being non-convex, our cost functional has no spurious local minima and the global minima is again the oracle solution, thereby providing the first method which is guaranteed to find this point for reasonable levels of noise, without resorting to combinatorial methods.


\end{abstract}

\section{Introduction}\label{sec:intro}
\subsection{Background}

We consider the classical compressed sensing problem of minimizing the cardinality $\card(x)=\|x\|_0$ of an approximate solution to an underdetermined equation system $Ax=b$, i.e. \begin{equation}\label{l0prob}\argmin_{x:~\|Ax-b\|_2<\eta}\card(x),
\end{equation} where $\eta>0$
is some allowed tolerance of the error and $x_0$ lies in $\R^n$ or $\C^n$. Problem \eqref{l0prob} is NP-hard \cite{natarajan1995sparse} and a popular approach is to replace $\card(x)$ with the convex function $\|x\|_1$, i.e. \begin{equation}\label{l1prob}\argmin_{x:~\|Ax-b\|_2<\eta}\|x\|_1.\end{equation} This method goes back (at least) to the 70's (see the introduction of \cite{candes2008enhancing} for a nice historical overview) but received increasing attention in the late 90's due to the work by Chen, Donoho and Saunders \cite{chen2001atomic} on what they called \emph{basis pursuit}, which amounts to solving
\begin{equation}\label{l1probdual}\argmin\left\{\lambda\|x\|_1+\frac{1}{2}\|Ax-b\|_2^2\right\}\end{equation}
for a suitable choice of parameter $\lambda$, playing the role of $\eta$ in \eqref{l1prob}. In fact, \eqref{l1probdual} is the dual problem of \eqref{l1prob} in the sense that for each $\eta$ there is a $\lambda$ such that the solution of \eqref{l1prob} and \eqref{l1probdual} coincides. The method received massive attention after the works of Donoho, Cand\'{e}s and coworkers in the early 2000, and the term compressed sensing was coined. In \cite{candes2006stable}, Cand\'{e}s, Romberg and Tao proved the surprising result that, given a $k$-sparse vector $x_0$ and a measurement
\begin{equation}
\label{nise}b=Ax_0+\ep,
\end{equation}
where $\ep$ is Gaussian noise, solving \eqref{l1prob} yields (for a suitable choice of $\eta$) a vector $\hat x$ that satisfies \begin{equation}\label{crt}
\|\hat x-x_0\|_2<C_k\|\ep\|_2,
\end{equation}
where $C_k$ is a constant. Arguing that it is impossible to beat a linear dependence on the noise (even knowing the true support of $x_0$ a priori), the estimate \eqref{crt} led the authors to conclude that ``no other method can significantly outperform this''. The result holds given certain assumptions on the matrix $A$, related to the Restricted Isometry Property (RIP) of $A$, which in a separate publication (Theorem 1.5, \cite{candes2005decoding}) was shown to hold with ``overwhelming probability''.

These results give the impression that the theory is more or less complete and that improvements only can be marginal. However, what is not so well known is that the mentioned results usually do not apply to regular applications of the framework. For example,
that ``statement $A(n)$'' holds with ``overwhelming probability'' only entails that the probability of $A(n)$ being false decays exponentially with the size $n$ of the application, hence a statement can hold with ``overwhelming probability'' and at the same time be false for most moderately sized applications (in some applications the dimension of the signals, in our case \(n\), is modest. See for instance \cite{feng-au-valaee-tan-2010}. For a more extensive survey on Compressed Sensing applications, see \cite{qaisar-bilal-iqbal-naureen-lee-2013}). In addition other assumptions need to be fulfilled for the ``overwhelming probability''-results to kick in, for example Theorem 1.5 in \cite{candes2005decoding} requires (according to the text below the theorem) that $k/n$ is of magnitude $10^{-4}$, which rules out most applications independent of whether $n$ is large or not. While the theory has been improved since 2005, the main problem that the results often do not apply to standard applied settings, remains. For example the recent works~\cite{adcock2016generalized,adcock2017breaking,candes-etal-acm-2011} provide \textit{asymptotic} theorems about when compressed sensing works in concrete setups. Moreover, in \cite{wainwright2009sharp} it was even shown that the method fails with probability tending to 1, as a function of $n$, if the ratio of $k,n$ and $m$ (the amount of measurements) is held fixed.

In addition to all this, whereas very strong recovery results were reported e.g.~in \cite{candes2006robust,chartrand2007exact,donoho2006most} for the case of \textit{exact data} $b=Ax_0$, in the presence of noise the method gives a well known bias (see e.g. \cite{fan2001variable,mazumder2011sparsenet}). The $\ell^1$ term not only has the (desired) effect of forcing many entries in $x$ to 0, but also the (undesired) effect of diminishing the size of the non-zero entries. This is clearly visible even in the one-dimensional situation; the function $\R\ni x\mapsto \lambda|x|+\frac{1}{2}|x-x_0|^2$ has its minimum shifted towards 0 from the sought point $x_0$. This has led to a large amount of non-convex suggestions to replace the $\ell^1$-penalty, see e.g. \cite{attouch2013convergence,blumensath2008iterative,blumensath2009iterative,bredies2015minimization,breheny2011coordinate,candes2008enhancing,chartrand2007exact,fan2001variable,fan2004nonconcave,fan2014strong,loh2013regularized,loh2017support,mazumder2011sparsenet,pan2015relaxed,selesnick2017sparse,wang2014optimal,zhang2012general,zou2008one}.
However, among these there is no clear winner and still $\ell^1-$methods seems to be the standard choice among engineers, maybe also due to its simplicity.
A fairly well-known non-convex alternative is the Minimax Concave Penalty (MCP) by Zhang, which was coined \textit{nearly unbiased} since the results in  \cite{zhang2010nearly} imply that the method does find the \textit{oracle solution} with probability tending to one under the assumptions of that paper. The ``oracle solution'' is sort of the holy grail of compressed sensing, and aside from Zhang's work and this publication, there seems to be no reliable methods (with proofs) of how to find it.

\subsection{Quadratic envelopes}\label{secQ}
In this paper we analyze two different methods to find the oracle solution, one which actually coincides with Zhang's MCP-penalty and a more intricate (and reliable) one that assumes a priori knowledge of the sparsity level $k$. In fact, these two are the tip of an iceberg of possible methods based on the ``Quadratic Envelope'', which we now introduce. Consider the general problem of minimizing \begin{equation}\label{genprob}
\K(x)=f(x)+\|Ax-b\|_2^2\end{equation} where $f$ is some non-convex penalty and $x$ is a vector in some linear space, not necessarily $\R^n$.
%\footnote{After much debating, we have decided to omit the traditional factor $\frac{1}{2}$ in front of the quadratic term, since it simplifies formulas substantially. More on this in Section \ref{regS}.}
The standard non-convex example mentioned in most introductions to papers on compressed sensing is $f(x)=\mu\card(x)$ for some trade off parameter $\mu>0$.
However, if the desired cardinality $k$ is known a priori, we can take $f$ to be the indicator function $\iota_{P_k}$ of the set $P_k=\{x:\card (x)\leq k\}$ in which case \eqref{genprob} reduces to
\begin{equation}\label{agt1intro}
\argmin_{\card (x)\leq k}\|Ax-b\|_2.
\end{equation}

In \cite{carlsson2018convex} \emph{quadratic envelope} $\Q_2(f)$ was introduced, where $\Q_2$ is the \emph{quadratic biconjugate} and \( f : \mathcal{V} \to \mathbb{R} \cup \{ \infty  \} \) can be any functional on a separable Hilbert space \( \mathcal{V} \); apart from the name, this transform was introduced already in \cite{carlsson2016convexification} and goes back to the work of Larsson, Olsson \cite{larsson-olsson-ijcv-2016}. It is defined as \begin{equation}\label{Qdef} \Q_2 (f)(x) = \sup_{\alpha\in \R,~y\in\mathcal{V}} \{\alpha-\|x-y\|^2:~ \alpha-\|\cdot-y\|^2\leq f \} \end{equation} see Figure \ref{f456} (taken from \cite{carlsson2018convex}) for an illustration. An explicit form for \( \Q_2 \) is not always possible, but it is for the two functions that this paper examines, see \eqref{l0000} and \eqref{lust1}. The quadratic envelope has also the property that $\Q_2 (f) (x)+\|x\|_2^2$ is the lower semi-continuous convex envelope of $f(x)+\|x\|_2^2$. The relationship between
\begin{equation}\label{genprobS}
\K_{reg}(x)=\Q_2(f)(x)+\|Ax-b\|_2^2\end{equation}
and the original functional in \eqref{genprob} was investigated in \cite{carlsson2018convex}. The inequality $\K_{reg}\leq \K$ is immediate by \eqref{Qdef} and will be used throughout.

\begin{wrapfigure}{r}{0.4\textwidth}
	\includegraphics[width=0.4\textwidth]{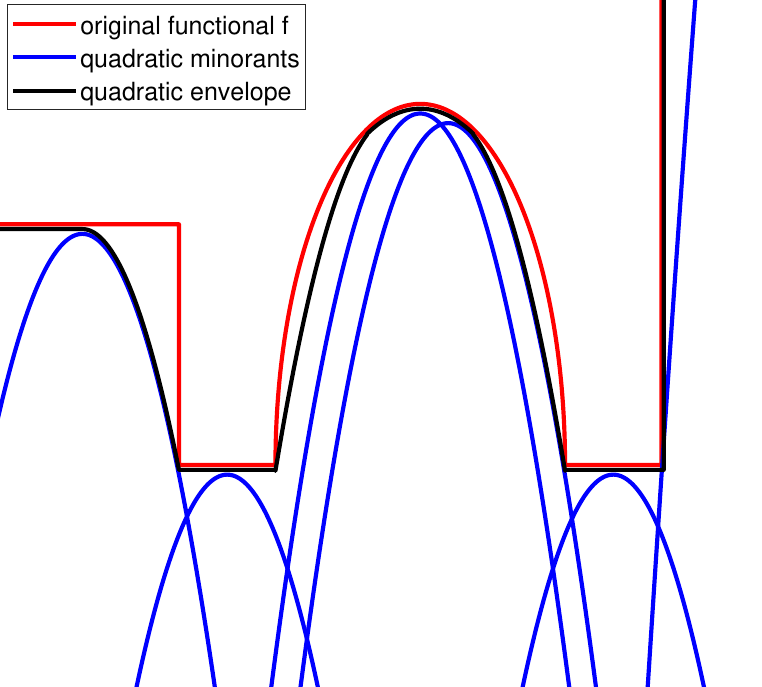}
	\caption{Illustration of a non-convex function $f$ (red) and its quadratic envelope $Q_2(f)$ (black). The black graph lies slightly below for illustration only.}\label{f456}
\end{wrapfigure} Given that $\|A\|_{\text{op}}<1$ (which always can be achieved by rescaling), the main result of \cite{carlsson2018convex} is that the set of local minimizers to \eqref{genprobS} is a subset of the local minimizers of \eqref{genprob}, and most importantly that the global minimizers coincide.\footnote{For the functionals considered in this paper the condition $\|A\|_{\text{op}}<1$ can be substantially relaxed, as we will explain further below.}

In the particular case of $f(x)=\mu \card(x)$, which is the first instance considered in this paper,  the functional \eqref{genprobS} has previously been introduced by Zhang \cite{zhang2010nearly} under the name Minimax Concave Penalty (MCP) and independently by Aubert, Blanc-Feraud and Soubies \cite{soubies-etal-siims-2015} under the name $CE\ell0$. It also shows up in earlier publications, for example (2.4) in \cite{fan2001variable}, but it seems like \cite{zhang2010nearly} is the first comprehensive performance study and  \cite{soubies-etal-siims-2015} the first publication where the connection with convex envelopes appears. For this choice of $f$, the value of the contributions of the present paper is mainly theoretical, which goes much beyond what was previously known. In particular we show that the global minimizer with the MCP-penalty (i.e.~$\Q_2(\card)$) is the oracle solution (for an appropriate range of the parameter $\mu$); hence it follows that the MCP is actually \textit{unbiased,} not merely \textit{nearly unbiased} as claimed in \cite{zhang2010nearly}.

To clarify what we mean by this, we note that it is easy to prove that the error in the oracle solution depends linearly on the noise, and hence the expectation of the error will be zero as long as the expectation of the noise is zero. In this sense any method finding the oracle solutions will be unbiased, which justifies the title of the paper.

The second penalty under consideration in this paper, $\Q_2(\iota_{P_k})$, is a new object that has only appeared previously in earlier publications by the authors of the present article. It also has the capacity of finding the oracle solution and the benefit that it does not rely on an appropriate parameter choice $\mu$, as long as the model order is known.
In contrast to the MCP-penalty (and most other previously studied sparsity priors) it is not separable but assigns a penalty that depends on the number of non-zero elements.
This leads to significant differences with respect to the optimization landscape and the distribution of stationary points that we will study in this paper.
In this article we provide theoretical results of the type \eqref{crt} for the two concrete functionals $\Q_2(\card)$ and $\Q_2 (\iota_{P_k})$. A more extensive discussion of previous results concerning MCP/$CE\ell0$ is found in Section \ref{review}, as well as other related results on non-convex optimization.

\subsection{Contributions}

A clear drawback with non-convex optimization schemes is that algorithms are bound to get stuck in local minima, and in concrete situations it is hard to determine whether this is the case or not. In the present article we give simple conditions which imply that the global minima of \eqref{genprobS} for $f(x)=\mu\card(x)$ \textit{is the oracle solution}, and moreover that any local minima  necessarily has a high cardinality unless it is the global minima. Hence, if a sparse local minima is found one can be sure that it is the oracle solution. In the case of $f=\iota_{P_k}$ we take this one step further and give conditions under which \eqref{genprobS} has a unique local minimizer, which hence must be the oracle solution and also the solution to the original problem \eqref{agt1intro}.

To be more precise, when the ``measurement'' $b$ has the form $b=Ax_0+\ep$ and $x_0$ is a sparse vector, we significantly improve the state of the art in compressed sensing in a number of ways. Firstly, the conditions on $A$ hold in greater generality, in the sense that our counterpart to conditions such as ``small Restricted Isometry Property-values'' or ``small mutual coherence'' (see e.g. \cite{foucart2013invitation}) hold to a much greater extent than existing theory for other approaches such as $\ell^1$-minimization or Iterative Hard Thresholding (IHT). Secondly, since the global minimizer of our functionals is the oracle solution, we obtain an estimate corresponding to \eqref{crt} where the involved constants are significantly smaller than $C_k$ (or other constants with a similar role found in the references). Thirdly, we show numerically that Forward-Backward Splitting (FBS) finds this in scenarios when competitors fail, thereby providing novel robust completely unbiased algorithms for compressed sensing (at least in the setting when $A$ has normalized Gaussian random columns). %For example, in Section \ref{sec:num} we show that we can find the oracle solution for $SNR$ as low as 4, using the simple FBS algorithm, for the case of a \( 100 \times 200 \) matrix $A$.

In Section~\ref{sec:mainResults} we present highlights from the theory, show some numerical results and compare with the traditional $\ell^1$-method \eqref{l1probdual}. In \ref{review} we give a brief review of the field. The remainder of the paper, Sections \ref{sec:unique}-\ref{seccardfix}, are devoted to developing the theory.

\section{Main Results and Innovations}\label{sec:mainResults}
Again, we will investigate minimizers of \eqref{genprobS} for the two  penalties $\Q_2(\mu\card)$ and $\Q_2(\iota_{P_k})$. We present key findings in sub-sections \ref{gtr} and \ref{kne}. First we give a brief review of the field.

\subsection{Brief review of related results}\label{review}

%As previously mentioned, $\Q_2(\mu\card)$ was introduced previously in \cite{zhang2010nearly} (MCP) and \cite{soubies-etal-siims-2015} (CE$\ell 0$), and appeared earlier also e.g.~in \cite{fan2001variable}, although in this paper they move on to introduce yet another penalty called SCAD which we discuss further below.

Needless to say, we are not the first group to address the shortcomings of traditional $\ell^1$-minimization by use of non-convex penalties. In fact, even before the birth of compressed sensing, the shortcomings of $\ell^1$-techniques were debated and non-convex alternatives were suggested, we refer to \cite{fan2001variable} for an overview of early publications on this issue. Moreover, shortly after publishing the celebrated result \eqref{crt}, Cand\'{e}s, Wakin and Boyd suggested an improvement called ``Reweighted $\ell^1$-minimization'' \cite{candes2008enhancing} which also became a big success. They provide a theoretical understanding of this algorithm as minimizing the non-convex functional $$f(x)=\sum_j \log(\epsilon+|x_j|)$$ where $\epsilon$ is a parameter chosen by the user. Figure \ref{S22card} shows the functions $\card(x),$ $|x|$ and $\log(0.1+|x|)-\log(0.1)$ as well as $\Q_2(\card)$. As is clear to see, $\log(0.1+|x|)-\log(0.1)$ is closer to $\card(x)$ than $|x|$, which may explain the better performance by reweighted $\ell^1$-minimization reported in \cite{candes2008enhancing,carlsson2020perfect}.
The functional $\Q_2(\card)$ is even closer to $\card(x)$, and while this certainly is one reason behind the superior theoretical results reported in this paper, there is still the issue of getting stuck in stationary points. In \cite{soubies-etal-siims-2015} the authors provide a macro algorithm to avoid non-local minima. In the same vein, Zhang \cite{zhang2010nearly} proposes to iteratively update relevant parameters to reach the desired global minima with higher probability.

Favorable results for $\Q_2(\mu \card)/MCP$ were reported in the recent paper \cite{loh2017support}, which compares the use of MCP with $\ell^1$ and reweighted $\ell^1$ (called LSP in \cite{loh2017support}) as well as SCAD (introduced in \cite{fan2001variable} which has similar performance as MCP). The numerical results in this paper seems also to reconfirm this, despite not employing any algorithm ensuring that we do not converge to an undesired stationary point. %At first glance, this seems to contradict the findings of \cite{soubies-etal-siims-2015} reported above. However, in their experiments they do not use $b=Ax_0$ for some sparse $x_0$ and they also do not use a matrix $A$ with good LRIP-properties.

The first theoretical justification of using MCP/$\Q_2(\mu\card)$ is Corollary 1 of \cite{zhang2010nearly}, which roughly speaking contains an algorithm which finds the global minimum of MCP with high probability, and shows that the probability that this differs from the oracle solution is low. The result is based on very technical assumptions involving constants $c_*,~c^*,~d^*,~d^o,~\gamma,~\sigma,~w^o,~\beta_*$ and $\tilde p_1$, and so it seems hard to verify if this result applies in a concrete situation.

A more recent theoretical justification to support the use of MCP is given in \cite{loh2017support} which, under a number of assumptions, prove that \eqref{genprobS} with $\Q_2(\mu \card)$ does have the oracle solution as a unique stationary point with high probability, and provide an estimate of the type \eqref{crt}, see Corollary 1. However, as with the results of Zhang, this result relies on a number of constants whose values are difficult to estimate, so it is hard to know when exactly the theorem applies. {\color{red} In addition we note that in many practical cases the MCP formulation has local minima, see our experimental evaluation, indicating that the assumptions made to ensure uniqueness are very restrictive.}
We believe that the corresponding theory in the present paper is much more transparent, with conditions that are more general and comparatively easy to verify, as well as stronger conclusions. We postpone further discussion of this to Section \ref{what}.

The papers \cite{attouch2013convergence,blumensath2008iterative,blumensath2009iterative,nikolova2013description,nikolova2016relationship} considers \eqref{genprob} for the cases $f(x)=\card(x)$ as well as $f(x)=\iota_{P_k}(x)$, and \cite{attouch2013convergence} show in particular that the FBS-algorithm applied to \eqref{genprob} converges to a stationary point, but a further analysis of this point is not present. In fact, it seems to us that these papers fail to recognize that the oracle solution often is the global minimizer of both \eqref{q1} and \eqref{q2}, which follows from the results of this paper (see Corollaries \ref{cor:doctorgadget} and \ref{cor:doctorgadget2} respectively).

Many other non-convex penalties have been proposed over the years \cite{selesnick2017sparse,bredies2015minimization,chartrand2007exact,pan2015relaxed,zou2008one,fan2004nonconcave,wang2014optimal,loh2013regularized,fan2014strong,zhang2012general,
zhang2010nearly,loh2017support,candes2008enhancing,breheny2011coordinate,fan2001variable,mazumder2011sparsenet}, and we make no attempt to review them here. The introduction of \cite{loh2017support} contains a recent overview. A common denominator seems to be that the penalty function is separable, i.e.~has the form $p(x)=\sum_j p_j(x_j)$ where $p_j$ are functions on $\R$ (except the recent contribution \cite{selesnick2017sparse}). The penalty $\Q_2(\iota_{P_k})$ is not of this form. In fact, $\Q_2(\iota_{P_k})$ is the simplest of a vast field of possible penalties introduced in \cite{larsson-olsson-ijcv-2016} that can be more tailormade to the problem at hand, neither of which is separable.

\subsection{Sparse recovery via $\Q_2(\mu\card)$}\label{gtr}

We return to the first problem of minimizing \eqref{genprobS} for $f=\mu\card(x)$ i.e.
\begin{equation}\label{q1}\kfi(x):=\mu \card{(x)}+\|Ax-b\|_2^2\end{equation} where the parameter $\mu$ controls the tradeoff between sparsity and data-fit. Motivated by Section \ref{secQ} we propose to regularize $\kfi$ with \begin{equation}\label{q1reg}\kfir(x)=\Q_2(\mu\card){(x)}+\|Ax-b\|_2^2.\end{equation} The graph of $\Q_2(\card)$ is depicted in Figure \ref{S22card}.

\begin{wrapfigure}{r}{0.4\textwidth}
 	\begin{center}
 		\includegraphics[width=.4\textwidth]{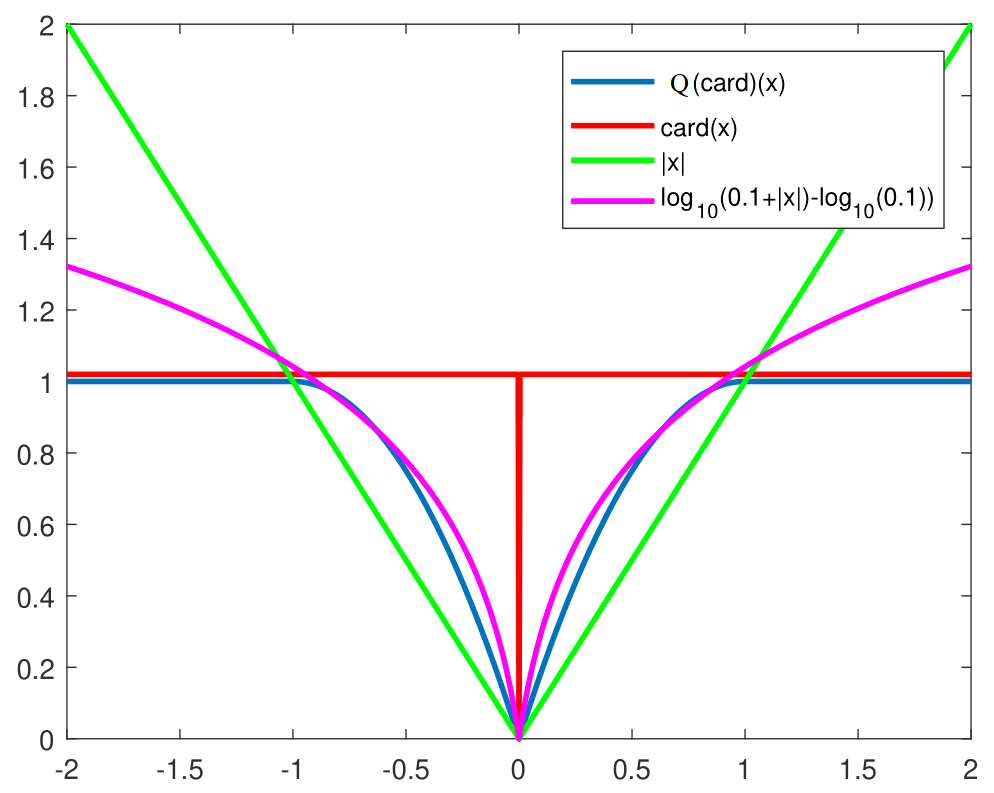}\\
 	\end{center}
 	\caption{Illustration of penalties.}
 	\label{S22card}
 \end{wrapfigure}
We will study uniqueness of sparse minimizers of both \eqref{q1} and \eqref{q1reg}, in the sense that we give concrete conditions such that if there exists one local minimizer $x'$ of \eqref{q1reg} with the property that $\card(x') \ll m$ (in a manner to be made precise), then
\begin{itemize}
\item $x'$ is automatically a global minimizer and also a solution to \eqref{q1}
\item any other stationary point $x''$ of \eqref{q1reg} satisfies $\card(x'') \gg \card(x')$.
\end{itemize}

To state our results, we remind the reader that $A$ satisfies a Restricted Isometry Property for integer $k$, if any $k$ columns of $A$ behaves approximately as an isometry, in the sense that \[ (1-\delta_k) \|x\|^2 _2 \le \| Ax\|_2 ^2 \le (1 + \delta_k) \|x\|^2 _2  \] for all \(k\)-sparse vectors \( x \in \mathbb{R}^n  \), \(k \in \mathbb{N}\), and some constant \( 0 \le \delta_k < 1   \). 
Classical results from compressed sensing literature usually require that the numbers $\delta_k$ are small, something which we have found is hard to fulfill in practice. For example, the famous estimate \eqref{crt} holds under the assumption that $\delta_{3k}+3\delta_{4k}<2$. This condition was later improved to the simpler estimate \begin{equation}\label{ki}\delta_{2k}<\sqrt{2}-1 \approx 0.4,
	\end{equation} (see \cite{candes2008restricted}) which is the estimate currently reproduced in textbooks on the subject, such as \cite{foucart2013invitation}. Our numerical evaluation (see Section \ref{sec:size}) shows that this condition is usually not satisfied for a Gaussian random matrix $A$ (with normalized columns) of size $100\times 200$ (a common size for many applications), except for $k=1$. The statement that RIP holds with overwhelming probability \cite{candes2005decoding} is therefore somewhat misleading.

% That the RIP-conditions are hard to satisfy in practice is well known by the community, and has led to interesting new contributions about efficiency of $\ell^1$ without RIP, given that the problem is sampled in a certain way, see e.g.~\cite{adcock2016generalized,adcock2017breaking}. However, these results are asymptotic in nature and do not apply in as general situation as the ones we will present here.

We base the theory of this paper on the Lower Restricted Isometry Property (LRIP), basically constituting the lower estimate of the RIP (introduced in \cite{blanchard-cartis-tanner-2011}). More precisely, we define \begin{equation}\label{beta}1-\delta_k^-=\inf\left\{\frac{\|Ax\|_2^2}{\|x\|_2^2}:~ x\neq 0,~\card(x)\leq k\right\}\end{equation} for $k=1\ldots n$. We say that $A$ satisfies LRIP with respect to the property $P_k=\{x:\card(x)\leq k\}$ if $\delta_k^-<1$. In other words $A$ is LRIP with respect to this property if and only if any $k$ chosen columns of $A$ are linearly independent. Clearly $\delta_k^-\leq \delta_k$ and for Gaussian matrices inequality typically holds, which is further discussed in Section \ref{sec:size}.

To give the reader an early insight into key findings, we state a simplified version of the main result of Section \ref{seccard}, Theorem \ref{thm:doctorgadget} (for the particular case $N=2k$).

\begin{corollary}\label{cor:doctorgadget}
Suppose that $A$ has columns in the unit ball of $\R^n$ or $\C^n$, that $b=Ax_0+\ep$ and set $\card (x_0)=k.$ Assume that
 the noise is small enough that the open interval $$\left(\frac{\|\ep\|_2}{1-\delta_{2k}^-},\frac{(1-\delta_{2k}^-)\min_{j\in\supp x_0}|x_{0,j}|}{2}\right)$$ is non-empty. Then for any $\mu$ with $\sqrt{\mu}$ in the above interval, we have that\begin{itemize}\item[$a)$]  Then there exists a unique global minimum $x'$ to $\kfir$ as well as $\kfi$, and it is the oracle solution.\item[$b)$] We have that $\supp x'=\supp x_0$ \item[$c)$]$$\|x'-x_0\|_2\leq \frac{\|\ep\|_2}{\sqrt{1-\delta_k^-}},$$ \item[$d)$] $\card(x'')> k$ for any other stationary point $x''$ of $\kfir$.\end{itemize}
\end{corollary}
Moreover, if the above estimates hold for some $N\gg 2k$ we can state that $x''$ has cardinality higher than $N-k$. In other words, either the algorithm finds the oracle solution or one with substantially higher cardinality. Although the theorem gives conditions on how to pick $\mu$, the involved quantities are generally not exactly known. For some matrix families good estimates exist e.g. \cite{blanchard-cartis-tanner-2011}. For other problems one has to proceed by trial and error (as with all to us known CS-methods). However, in our experience, the method is very robust and finds the oracle solution for a range of $\mu$-values, as opposed to e.g.~traditional $\ell^1$-minimization \eqref{l1probdual} which gives a different solution for each $\lambda$.

Note that the conditions on ``noise'' $\ep$ and ``ground truth'' $x_0$ are very natural; if the noise is too large or if the non-zero entries of $x_0$ are too small, there is no hope of correctly retrieving the support. Also note the absence of a condition forcing $\delta_{2k}^-$ to be ``small'', in sharp contrast to other results in the field such as \eqref{ki} or $\delta_{3k}<1/\sqrt{32}$ in \cite{blumensath2008iterative} (conditions that are very hard to satisfy, see Section \ref{sec:size}). On the contrary, as long as $\delta_{2k}^-<1$, Corollary \ref{cor:doctorgadget} holds, and in order for it to apply for some $\mu$ one needs that the signal to noise ratio, measured as \begin{equation}\label{snr}SNR=\frac{\min_{j\in\supp x_0}|x_{0,j}|}{\|\ep\|_2},\end{equation} has to be sufficiently large, (more precisely larger than $\frac{2}{{(1-\delta_{2k}^-)}^2}$, for then the interval in the corollary is non-void).

\subsection{Sparse recovery via $\Q_2(\iota_{P_k})$.}\label{kne}

We now discuss the situation when the model order, i.e. the amount $k$ of non-zero entries, is known. This problem is also known as the $k$-sparse problem and studied e.g.~in \cite{blumensath2008iterative}. For simplicity we restrict attention to $\R^n$, corresponding results for $\C^n$ are similar but the assumptions on $A$ are slightly more technical (see Section \ref{Kfeas}). As pointed out earlier the NP-hard problem \eqref{agt1intro} can be written
\begin{equation}\label{q2}\K_k(x)=\iota_{P_k}(x)+\|Ax-b\|_2^2\end{equation} (where the subindex $k$ separates the notation from \eqref{q1}) which we regularize with \begin{equation}\label{q2reg}\K_{k,reg}(x)=\Q_2(\iota_{P_k})(x)+\|Ax-b\|_2^2.\end{equation}
{\color{red}Figure~\ref{fig:iota} shows $\Q_2(\iota_{P_1})$ as a function of two variables (in the positive quadrant). The penalty assigned is zero for all vectors with no more than one non-zero variable.
For comparison we also plot $\Q_2(\card)$. Note that $\Q_2(\card)$ is constant in the region where both variables are larger than $\sqrt{\mu}$. This shape makes it likely that $\K_{reg}$ has local minimizers of high rank. In contrast $\Q_2(\iota_{P_1})$ has large gradients in this area which as we shall se makes it possible to exclude such stationary points for $K_{k,reg}$.

\begin{figure}[hbt]
	\begin{center}
		\includegraphics[width=.4\textwidth]{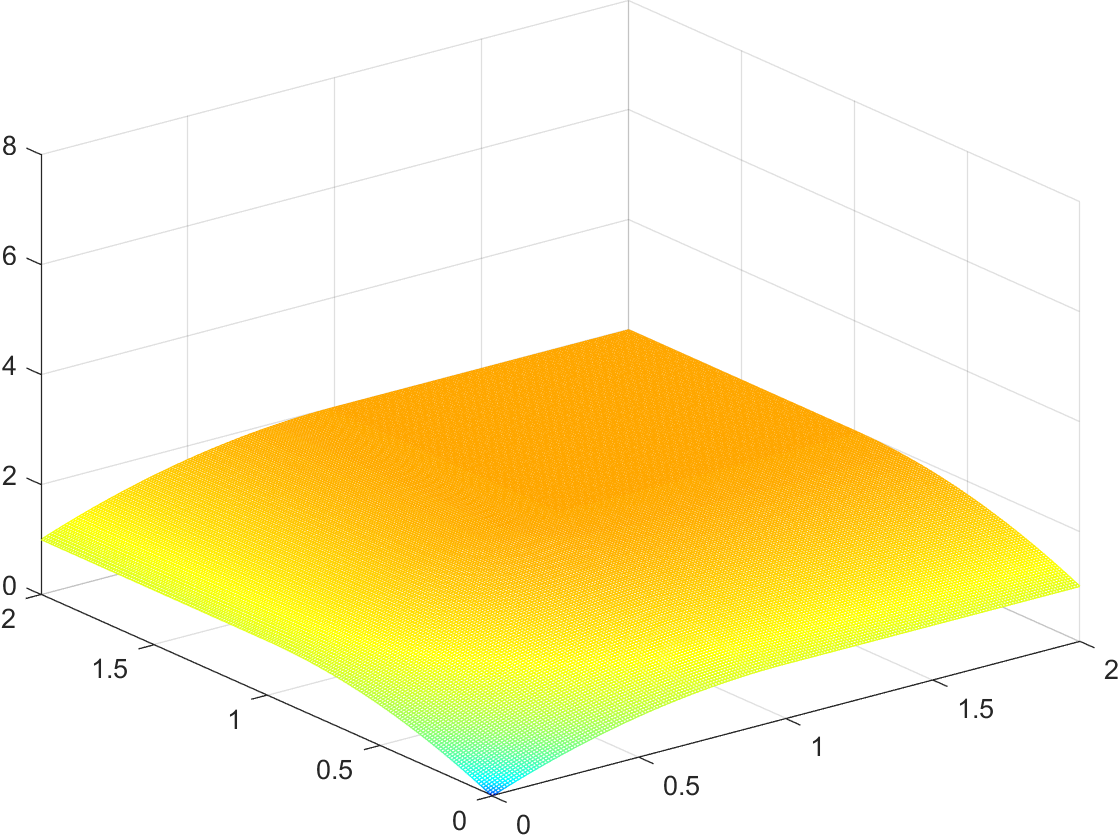}
		\includegraphics[width=.4\textwidth]{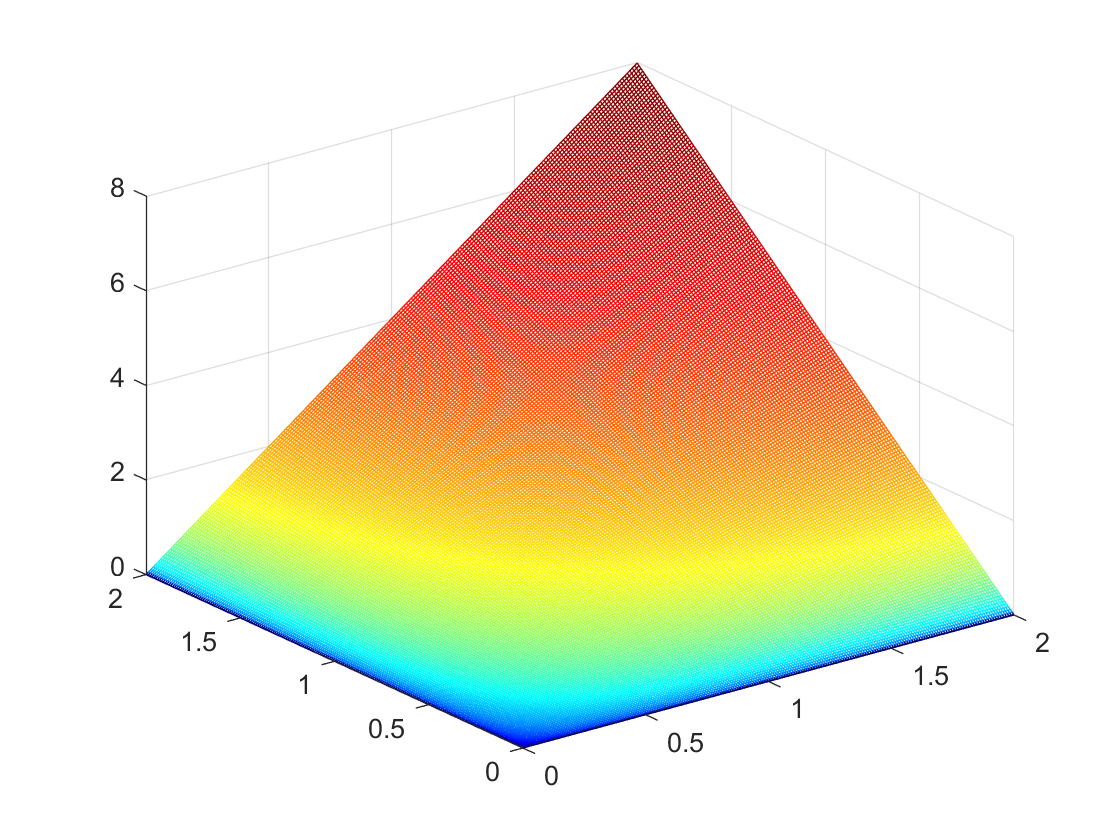}		
	\end{center}
\caption{Two dimensional illustrations of the functions $\Q_2(card)$ (left) and $\Q_2(\iota_{P_1})$ (right).}
\label{fig:iota}
\end{figure}
}
We first present a result where $b$ is not necessarily given by $Ax_0+\epsilon$.
\begin{corollary}\label{cytwsgdf}
Let $A$ have columns in the unit ball such that no pair is orthogonal, and assume that $n\geq m+k+2$. Any local minimizer $x'$ of $\ksir$ then satisfies $\card(x')\leq k$. Moreover, set $z'=(I-A^* A)x' + A^* b$, let $\tilde z'$ contain the elements of $z'$ sorted by decreasing magnitude,
 and assume that
\begin{equation}\label{n}|\tilde z'_{k+1}|<(1-2\delta_{2k}^-)|\tilde z'_{k}|.\end{equation}
Then $x'$ is the unique global minimum of $\ksi$ and $\ksir$.
\end{corollary}
A similar result also holds in the situation of the previous section. The interesting point to note is that there is a simple verifiable condition on whether a solution to  \eqref{agt1intro} has been found, given that some estimate of ${\delta_{2k}^-}$ is available (see e.g. Theorem 9.26 of \cite{foucart2013invitation} or \cite{blanchard-cartis-tanner-2011}).

 Corollary \ref{cytwsgdf} is a combination of Theorem \ref{celok6} and \ref{thm:globalpoint2f}.  %We remark that in the typical compressed sensing application, $A$ is a matrix with $m<<n$ and $k<<m$. If the columns of $A$ are normalized random, then the conditions on $A$ are satisfied with probability 1. Moreover, any subset of $2k$ columns will be close to an isometry as long as $2k<<m$, so it is not unreasonable to expect that $\beta_{2k}\approx 1$. In this case the assumption \eqref{n} is quite reasonable since $|\tilde z'_{k+1}|\leq |\tilde z'_{k}|$ by construction and $2(1-\delta_{2k}^-)-1\approx 1.$ The size of $\beta_{2k}$ in the above scenario is further discussed in subsection \ref{sec:size}.
We now consider the case when $b=Ax_0+\ep$ and we wish to retrieve $x_0$, where $\card(x_0)=k$. By Theorem \ref{cor:dogadgetf}, we have (for $A$ as in the previous corollary);

\begin{corollary}\label{cor:doctorgadget2}
 Assume the SNR (as measured in \eqref{snr}) is greater than $\frac{3}{\sqrt{1-\delta_{2k}^-}}$. Then the oracle solution is a unique global minimizer $x'$ to $\ksir$ with $\supp (x')=\supp (x_0)$ and moreover it satisfies $\|Ax'-b\|_2\leq \|\epsilon\|_2$ and $$\|x'-x_0\|_2\leq \frac{\|\ep\|_2}{\sqrt{1-\delta_k^-}}.$$
\end{corollary}

Finally, if the SNR is also greater than $$\para{\frac{1}{1-2\delta_{2k}^-}+\frac{1}{\sqrt{1-\delta_k^-}}},$$ Corollary \ref{new} says that $\ksir$ has no local minimizers (except for the oracle solution). An interesting point to note is that minimizing $\ksi$ can be seen as finding one among $\binom{n}{k}$ possible minimizers (see the proof of Theorem \ref{cor:dogadgetf}). %These local minimizers are different with probability 1 (since the event of a coordinate being 0 has probability 0), so the original problem of minimizing $\ksi$ can be thought of as finding one solution among $\binom{n}{k}$ possible.
However, $\binom{n}{k}$ is typically a large number, for example if $k=10$ and $n=1000$ it is around $2\cdot 10^{23}$. The above corollary states that all but one of these, the relevant one, disappears when regularizing with $\Q_2(\iota_{P_k})$, which is rather amazing, in the authors humble opinion. However,
this clearly demands that $\delta_{2k}^-<0.5$, to be compared with the state of the art assumption $\delta_{2k}<\sqrt{2}-1\approx 0.4$ for when standard compressed sensing results kick in \cite{foucart2013invitation}.
In which situations is it likely to assume that either of these hold? We try to shed some light on this in the next section.

\subsection{On the size of RIP/LRIP-constants}\label{sec:size}

RIP-values are notoriously difficult to estimate, which makes it hard to compare theorems in compressed sensing. For example, the currently best known estimate for \eqref{crt} was proven in \cite{candes2008restricted}, and is reproduced in textbooks such as \cite{foucart2013invitation}. It says that $C_k=8.5$ if $\delta_{2k}=0.2$, but how likely is that to happen? In \cite{foucart2013invitation} very intricate estimates in this direction are give  in Theorem 9.27, which claims that the \(2k\)-RIP constant \( \delta_{2k}\)  of a random Gaussian matrix \( A/ \sqrt{m} \) is\footnote{Constructing the matrix like this gives expected value of the column norms equal to 1, so is very similar to normalizing the columns, as done in the examples of this paper}  \[ \le 2 \left(  1 + \frac{1}{\sqrt{2 \ln(e \cdot n /2k )}} \right) \eta +  \left(  1 + \frac{1}{\sqrt{2 \ln(e \cdot n /2k )}} \right)^2 \eta^2   \] with probability \( 1- \epsilon  \) if \[ m \ge 2 \eta^{-2} (2k \ln(e \cdot n/2k) + \ln(2\epsilon^{-1})    ).    \] Now let's suppose we are interested in a very sparse signal, \( k=10  \), and \( n=1000  \). Then \[ \left(  1 + \frac{1}{\sqrt{2 \ln(e \cdot 1000 /20 )}} \right) \approx 1.32  \] and $\left(  1 + \frac{1}{\sqrt{2 \ln(e \cdot 1000 /20 )}} \right)^2 \approx 1.74$. The equation $1.74 \eta^2 + 2 \cdot 1.32 \eta = c $ gives the positive solution \[ \eta (c) \approx (\sqrt{1.74 c + 1.32^2} - 1.32)/1.74.  \] For \( c= 0.2 \), \( \eta \approx 0.072 \). Therefore we would need \( m \ge 37878 \), independently on the probability degree \( \epsilon \); this is absurd since we would like \( m \ll n =1000 \).

Of course, there is the possibility that the estimates for $\delta_{2k}$ are poor and that the reality is different. To test this we computed values of $\delta_{j}$ and $\delta_{j^-}$ for matrices of various size. {\color{red}The test matrices where generated by first drawing elements from i.i.d Gaussian distributions and then normalizing the resulting columns. Note that this gives a matrix with columns drawn from a uniform distribution on the sphere.
The results presented below where averaged over 5 trials.}
\begin{table}[h!]
\begin{center}
 \caption{$m=25,~n=50$}
\begin{tabular}{|c|c|c|c|c|c|}
   \hline
   j: & 2 & 3 & 4 & 5 & 6\\ \hline
   % after \\: \hline or \cline{col1-col2} \cline{col3-col4} ...
   $\delta_j:$ & 0.66 & 1.04 & 1.36 & 1.65 & 1.88 \\
   $\delta_j^-:$ & 0.66 & 0.79 & 0.87 & 0.92 & 0.95 \\
   \hline
 \end{tabular}
\end{center}
\end{table}

From table 1 we can note several interesting things. For example, $\delta_j^-$ is usually a bit smaller than $\delta_{j}$, and whereas the latter can become larger than 1 the former can not, by definition. In fact, by the definition it is easy to see that $\delta_j^-=1$ if and only if there are $j$ linearly dependent columns in the matrix. This means that with probability 1, we always have $\delta_j^-<1$ for $j\leq m$ whereas $\delta_j^-=1$ for all $j>m$. In particular, Corollary \ref{cor:doctorgadget} is applicable with probability 1 whenever $k\leq m/2$.

The second thing to note is that we do not present very many values, which is related to the computational time. If we were interested in computing $\delta_{20}$ for a matrix with $n=1000$, as discussed initially, we would need to perform $\binom{1000}{20}\approx 4\cdot10^{41}$ SVD's. In fact, even computing $\delta_{7}$ for $n=50$ requires around $10^9$ SVD's, (which is not impossible but we skipped it since the numbers are very poor anyway). For this reason, the typical sizes of $\delta_j$'s remain a mystery, which likely is a reason behind the widespread belief that these numbers often are decent. To shed some light for larger matrices, we now compute for $j$ up to 4 and $m=100$ as well as $250$ (with $n=2m$).

\begin{table}[h!]
\parbox{0.45\linewidth}{\begin{center}
 \caption{$m=100,~n=200$}
\begin{tabular}{|c|c|c|c|}
   \hline
   j: & 2 & 3 & 4\\ \hline
   % after \\: \hline or \cline{col1-col2} \cline{col3-col4} ...
   $\delta_j:$ & 0.39 & 0.61 & 0.80 \\
   $\delta_j^-:$ & 0.39 & 0.52 & 0.62 \\
   \hline
 \end{tabular}
\end{center}}
\parbox{0.45\linewidth}{\begin{center}
 \caption{$m=250,~n=500$}
\begin{tabular}{|c|c|c|c|}
   \hline
   j: & 2 & 3 & 4\\ \hline
   % after \\: \hline or \cline{col1-col2} \cline{col3-col4} ...
   $\delta_j:$ & 0.28 & 0.43 & 0.55 \\
   $\delta_j^-:$ & 0.28 & 0.38 & 0.44 \\
   \hline
 \end{tabular}
\end{center}}

\end{table}

The most striking thing to note is that the numbers are still terribly poor, even for $m=250$. It certainly came as a surprise to the authors that none of the classical results on compressed sensing applies in the $250\times 500$ setting, unless $k=1$ and in this case the constant $C_k$ is approximately $14$ (based on our five trials average). Here a strength of the results of this paper becomes apparent, because even for an extremely poor value like $\delta_k^-=0.95$ we have that the constant in the error estimate $\|x'-x_0\|_2\leq \frac{1}{\sqrt{1-\delta_k^-}}\|\ep\|_2$ equals $4.5$, almost half the value you get for $C_k$ when $\delta_{2k}=0.2$ in \cite{candes2008restricted}, as reported initially.

%To be fair, it needs to be said however that for Corollary \eqref{cor:doctorgadget} to kick in we also need $\delta_{2k}^-<1$. To take a concrete example, with $k=3$, $m=100$ and $n=200$ the result is applicable whenever the SNR (as defined in \eqref{snr}) is above 800.

In fact, despite the difficulty in estimating the constants, it is not impossible to compare the quality of estimates. If we set $f_{C}(x)=\frac{4 \sqrt{1+x}}{1-(1+\sqrt{2})x}$ and $f_{CGO}(x)=\frac{1}{\sqrt{1-x}}$ then the constant $C_k$ in \eqref{crt}, as defined in \cite{candes2008restricted}, is given by $f_C(\delta_{2k})$ whereas the corresponding constant in Corollary \ref{cor:doctorgadget} and \ref{cor:doctorgadget2} is given by $f_{CGO}(\delta_k^-)$. The functions $f_C$ and $f_{CGO}$ are displayed in Fig \ref{RIPvsLRIP}. \begin{wrapfigure}{r}{0.5\textwidth}
	\begin{center}
		\includegraphics[width=.5\textwidth]{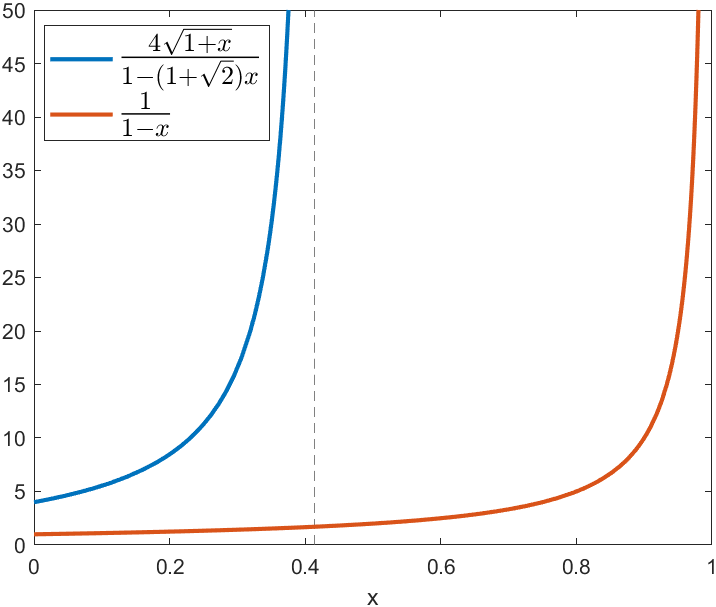}\\
	\end{center}
	\caption{$f_{CGO}$ in red and $f_C$ in blue.}
	\label{RIPvsLRIP}
\end{wrapfigure} Clearly the latter constant is vastly better by just comparing these graphs, and this conclusion is further strengthened by noting that $\delta_k^-\leq \delta_{k}\leq \delta_{2k}$.

A fourth thing to note from the tables is that the $\delta_k$'s do decrease with $m$, as predicted by the theory, and once we hit $m=250$ all reported numbers are below 0.5, the requirement for Corollary \ref{cor:doctorgadget2} to kick in. However, note that once Corollary \ref{cor:doctorgadget2} applies the error estimate immediately gets a very favorable constant, since $1/\sqrt{1-0.5}=\sqrt{2}\approx 1.4$. On the other hand, the $\ell^1$-results by \cite{candes2008restricted} still only applies for $k=1$ and then the constant $C_k$ equals 13.96.

How much better does it then get in the asymptotic regime? The best estimates of this we have found is in \cite{blanchard-cartis-tanner-2011}, which gives advanced probabilistic estimates as well as extensive numerical evaluations using sophisticated methods to estimate RIP/LRIP-values. In particular Figure 2.3 and 2.4 are enlightening, where it is shown that for $\frac{m}{n}=0.5$, one needs to have $k$ well below $1\%$ of $m$ to have any hope of achieving $\delta_{2k} = 0.4$, which is what is required in \eqref{ki}. More precisely, following \cite{blanchard-cartis-tanner-2011} we need $\mathcal{L}(0.5,\frac{2k}{m})$ and $\mathcal{U}(0.5,\frac{2k}{m})$\footnote{\(\mathcal{L}\) and \( \mathcal{U} \) are the asymptotic RIP bounds. Informally speaking, and here we quote \cite{blanchard-cartis-tanner-2011} verbatim, "For large matrices from the Gaussian ensamble, it is overwhelmingly unlikely that the RIP asymmetric constants \( L(k,m,n) \) and \( U(k,m,n) \) will be greater than \(\mathcal{L} (\delta,\rho) \) and \(\mathcal{U} (\delta,\rho) \)". \( \delta \) and \( \rho \) are such that \( n/N \to \delta  \) and \( k/n \to \rho \) as \( n \to \infty  \). \( L(k,m,n) \) is what we called \( \delta_k^{-} \) for an \( m \times n \) matrix and \( U(k,m,n) \) is its natural upper-counterpart.} to be below 0.4, which happens around $k/m\approx 1.5\cdot10^{-3}$. It is also clear that RIP-values are consistently higher with a notable difference. Based on this, it seems safe to conclude that for a large amount of settings where $\ell^1$-methods are used, there is very limited theoretical evidence for their applicability, at best.

%A more rigorous analysis of these values is performed in \cite{carlsson2020perfect}, still for the case $m=2n$, which indicate that $2\sqrt{k/m}$ is a good lower estimate for both $\delta_{2k}$ and $\delta_{2k}^-$, which actually are not that different, in the case of random Gaussian matrices with normalized columns. This is in line with estimates from random matrix theory, see e.g. Ch. 9 in \cite{foucart2013invitation}. The restriction $\delta_{2k}<0.4$ thus entails (roughly) that $2\sqrt{2k/m}<0.4$ or $k=0.8\%$ of $m$.
%The restriction $\delta_{2k}^-<0.5$ entails, with the same logic, that
%
%The left graph basically displays estimates of $\beta_{k}^2$, whereas the right displays estimates of $\delta_k$. The threshold for $\beta_{2k}$ to pass $1/\sqrt{2}$ seems to be $k/m\approx 0.05$. On the other hand, the most generous threshold for RIP-based results concerning $\ell^1$-methods to kick in seems to be \cite{candes2008restricted}, which is $\delta_{2k}< 1/\sqrt{2}$. Judging from Fig. 6 of \cite{carlsson2020perfect}, this is unlikely to happen unless $k/m<0.02$.

%Another nice feature of $\beta_j$ is the inequality $\beta_j\geq \sigma_j$ where $(\sigma_j)_{j}$ are the singular values of $A$ (Proposition \ref{p0}), which allows us to actually estimate the assumptions in Corollary \ref{cor:globalpoint2}-\ref{cor:doctorgadget2} in concrete situations, although the resulting bounds are likely to be unnecessarily restrictive. For this reason we included the singular values in the Figures ...

\subsection{What's in a theorem?}\label{what}

The so called ``oracle solution'', i.e. the one you would get if an oracle told you the true support $S$
 of $x_0$ and you were to solve the (overdetermined) equations system $A_Sx=b$ where $A_S$ denotes the $m\times k$ matrix whose columns are those with indices in $S$ (and then expand $x$ to $\R^n$ by inserting zeroes off $S$). This is clearly the best
possible solution one could hope for (as argued also in \cite{candes2006stable}).

Corollaries \ref{cor:doctorgadget} and \ref{cor:doctorgadget2} claim that the oracle solution is a unique global minimizer of the respective functional, not necessarily that a given algorithm will find this global minimizer. In our experience, working either with FBS or ADMM, the algorithms do find the oracle solution in very difficult scenarios when one initializes at zero\footnote{Initializing $\K_{reg}$ at the least squares solution is not good. As evidenced by our numerical evaluation in Section~\ref{sec:num} there seems to be many local minima nearby.}, but we do not have a proof for this. We can prove that FBS converges to a stationary point and that the stationary points in Corollary \ref{cor:doctorgadget2} are not local minima, except for the oracle solution, see Section \ref{sec:num}.

What is the value of these observations and how do they compare with the existing literature? For example, \cite{blumensath2009iterative} studies the minimization of \eqref{q2} itself (which, if we apply FBS, leads to Iterative Hard Thresholding for $k$-sparsity, denoted IHT$_k$), and it actually guarantees that IHT$_k$ converges to within $5\|\ep\|$ of the oracle solution. They do not prove that they've found the oracle solution, but combined with Corollary \ref{cor:doctorgadget2} it follows that this is indeed the case (for SNR's such that the Corollary applies).
On first sight this is a much stronger conclusion, since they actually prove that their algorithm avoids unwanted stationary points. However, the method performs much worse in practice, see Figure \ref{f4}. The difference lies in the fact that \cite{blumensath2009iterative} assumes that $\delta_{3k}<\frac{1}{\sqrt{32}}\approx 0.18$, whereas Corollary \ref{cor:doctorgadget2} applies as long as $\delta_{2k}^-<0.5$, which is much more easy to fulfill in practice.

The strength of a result not only in the conclusion, but in how much one needs to assume. For example, there are many papers giving conditions under which minimization of \eqref{l1probdual} or non-convex alternatives find the true support.
If we have a method that would find a vector $x'$ with the correct support $S$ (with a bias or not), we can always get this unbiased solution by simply discarding $x'$ and follow the above procedure to get the oracle solution. Therefore the issue of finding the support is maybe more central than having a good estimate of $\|x'-x_0\|_2$. Conditions under which LASSO finds the correct support are given e.g.~in \cite{wainwright2009sharp} and for a more general class of non-convex penalties in \cite{loh2017support}. In both cases however, the theorems involve constants whose size is unknown, and their applicability can not be verified in a concrete problem instance. To be more concrete, the latter paper does have a result claiming that MCP finds the oracle solution with given probability, but apart from involving conditions that are very difficult to verify, the conclusion contains the statement that MCP has a unique stationary point. This is rarely true (see e.g.~\cite{soubies-etal-siims-2015} or Figure \ref{f5} below) and hence it follows that these conditions are substantially more strict than ours.
Concrete theorems proving that LASSO finds the true support, such as those found in \cite{donoho2005stable} are rather weak, see \cite{carlsson2020perfect} which studies this topic in depth.

We therefore believe that our framework is more generally applicable even for the already well-studied MCP penalty, and that the relative simplicity and verifiability of the assumptions makes the theory attractive, in particular combined with superior performance numerically, at least in the standard synthetic setting (see Section \ref{sec:num}).
Moreover, the penalty $\Q_2(\iota_{P_k})$ is not separable, and the underlying ideas of this work are based on the Quadratic Envelope as a regularizer and extends to a whole class of more advanced sparsifying penalties, of which $\Q_2(\iota_{P_k})$ is merely one. We leave further extensions for future work, but remark that the whole machinery developed here can also be lifted to low rank matrix problems, see \cite{carlsson2019biased}.

%The setting of \cite{loh2017support} is rather different and we have not been able to verify reasonable values of the involved constants $c_l,c_u,c_\infty,R,\mu,\lambda, \eta, c_1,c_2,c_3$ and $\gamma$ in order to compare the strength of our Corollary \ref{cor:doctorgadget} with Corollary 1 in \cite{loh2017support}. We simply note that they point in the same direction and that Corollary \ref{cor:doctorgadget} holds under simpler conditions. The same remark goes for Theorem 3 in \cite{pan2015relaxed} and Theorem 1 in \cite{fan2001variable}, which provide conditions under which a class of non-convex optimization problems give the same estimate as the ``oracle estimator'', with a high probability.

\section{Uniqueness of sparse stationary points}\label{sec:unique}
We now turn to the heart of the matter, namely uniqueness of sparse minimizers of $\K_{reg}$ as defined in \eqref{genprobS}, where $f$ can be any function with values in $\R\cup\{\infty\}$.
We say that $x$ is a stationary point of a given function $g$ if \begin{equation}\label{defstationary}\underset{y\neq 0}{\liminf_{y\rightarrow x}}~ \frac{g(x+y)-g(x)}{\|y\|}\geq 0.\end{equation}
  If $g$ is a sum of a convex function $g_c$ and a differentiable function $g_d$ and we work in $\R^n$, it is not hard to see that $x$ is a stationary point if and only if $-\nabla g_d(x)\in \partial g_c(x)$ where $\partial g_c$ denotes the usual subdifferential used in convex analysis, and $\nabla g_d$ the standard gradient. The same is true in the complex case, i.e.~when working over $\C^n$, upon suitable modification of the concept of subdifferential and gradient. For convenience we provide the details in Appendix~\ref{app:prel}.

Set \begin{equation}\label{gdef}\G(x)=\frac{1}{2}\Q_2(f)(x)+\frac{1}{2}\fro{x},\end{equation} i.e. $2\G$ the l.s.c. convex envelope of $f(x)+\fro{x}$.
We have
\begin{equation}\label{fr3}\K_{reg}(x)=2\G(x)-\fro{x}+\fro{Ax-b}\end{equation}
which upon differentiation yields that $x'$ is a stationary point of $\K_{reg}$ if and only if
\begin{equation}\label{sp}(I-A^* A)x' + A^* b\in\partial \G(x').\end{equation}
since $\nabla\left(\fro{Ax-b}-\fro{x}\right)=2A^*(Ax-b)-2x$, see the appendix for details. Given any $x$, we therefore associate with it a new point $z$ via
\begin{equation}\label{y}
z=(I-A^* A)x + A^* b.
\end{equation}
This point will play a key role in this paper, in fact, it has already appeared in Corollary \ref{cytwsgdf}. Suppose now that $x'$ and $x''$ are two \textit{sparse} stationary points in the sense that $x''-x'\in P_N$ for some $N$ less than $m$.

\begin{proposition}\label{thm:statpoint}
Let $x'$ and $x''$ be distinct stationary points of $\K_{reg}$ such that $x''-x'\in P_N$. Then
\begin{equation}\label{ineq}\re\scal{z''-z',x''-x'}\leq \delta_N^-\fro{x''-x'}.\end{equation}
\end{proposition}

The above proposition will mainly be used backwards, i.e.~we will show that \eqref{ineq} does not hold and thereby conclude that $x''-x'\not\in P_N$.

\begin{proof}
We have $$z''-z'=( I-A^* A)x''+A^*b-( I-A^* A)x'-A^*b=(I-A^* A)(x''-x'),$$
so taking the scalar product with $x''-x'$ gives $$\re\scal{z''-z',x''-x'}=\fro{x''-x'}-\fro{A(x''-x')}\leq \delta_N^-\fro{x''-x'},$$ as desired. Note that it is not necessary to take the real part, but we leave it since scalar products in general can be complex numbers.
\end{proof}

As we shall see, the point $z'$ has a decisive influence on the coming sections. To begin with, it has the following interesting property.

\begin{proposition}
A point $x'$ is a stationary point of $\K_{reg}$ if and only if it solves the convex problem $$x'\in\argmin_x \Q_2(f)(x)+\fro{x-z'}.$$
\end{proposition}
Note the absence of $A$ in the above formula, which in particular implies that $\Q_2(f)(x)+\fro{x-z'}$ is the convex envelope of $f(x)+\fro{x-z'}$.
\begin{proof}
As noted in \eqref{sp}, $x'$ is a stationary point of $\K_{reg}$ if and only if $z'\in\partial\G(x')$. By the same token, $x'$ is a stationary point of $$\Q_2(f)(x)+\fro{x-z'}=2\G(x)-2\re\scal{x,z'}+\fro{z'}$$ if and only if $z'\in\partial\G(x')$, and since the functional is convex (and clearly has a well defined minimum) the stationary points coincide with the set of minimizers.
\end{proof}

\section{The sparsity problem}\label{seccard}

We return to the sparsity problem, and consider $f(x)=\mu\card(x)$ where $\mu$ is a parameter and $\card (x)$ is the number of non-zero entries in the vector $x$. In this case we have,
\begin{equation}\label{l0000}\Q_2(\mu\card)(x)=\sum_{j=1}^n \mu-\para{\max\{\sqrt{\mu}-{ |x_j|},0\}}^2.\end{equation}
To recapitulate, we want to minimize \eqref{q1}, i.e.~\begin{equation}\label{t57}\kfi(x)=\mu\card(x)+\fro{Ax-b}\end{equation} which we replace by \eqref{q1reg}, i.e. \begin{equation}\label{fr2}\kfir(x)=\Q_2(\mu\card)(x)+\fro{Ax-b}.\end{equation}

\subsection{Equality of minimizers for $\kfi$ and $\kfir$}\label{s4}

As noted by Aubert, Blanc-Feraud and Soubies (see Theorems 4.5 and 4.8 in \cite{soubies-etal-siims-2015}),
$\kfir$ has the same global minima and potentially fewer local minima than $\kfi$ if \begin{equation}\label{alpha}\|A\|_{\infty,col}=\sup_i \|a_i\|_2\leq 1,\end{equation} where $a_i$ denotes the columns of $A$. Below we (essentially) reproduce their statement in the terminology of this paper. A proof is included in the appendix for completeness.

\begin{theorem}\label{celok5}
If $\|A\|_{\infty,col}<1$, then any local minimizer of $\kfir$ is a local minimizer of $\kfi,$ and the (nonempty) set of global minimizers coincide. If merely $\|A\|_{\infty,col}=1$, then any global minimizer of $\kfir$ which is not a global minimizer of $\kfi$, belongs to a connected component of global minimizers which includes at least two global minima of $\kfi$.
\end{theorem}

\subsection{On the uniqueness of sparse stationary points}\label{erai}
%Given subset $S\subset\{1,\ldots,n\}$ we define $A_S$ to be the matrix whose columns are $a_j$ for $j\in S$ and $0$ else.
Next we take a closer look at the structure of the stationary points. Given $N$ such that $\delta_{N}^-<1$, we will show that under certain assumptions the difference between two stationary points always has at least $N$ elements.
Hence if we find a stationary point with less than $N/2$ elements then we can be sure that this is the sparsest one. The main theorem reads as follows:

\begin{theorem}\label{thm:statpoint:vec}
Let $x'$ be a stationary point of $\kfir$, let $z'$ be given by \eqref{y}, and assume that
\begin{equation}\label{e4}|z_i'|\not\in\left[{(1-\delta_N^-)}{\sqrt{\mu}},\frac{{\sqrt{\mu}}}{1-\delta_N^-}\right]\end{equation} for all $i\in\{1,\ldots,n\}$.
If $x''$ is another stationary point of $\kfir$ then $$\card(x''-x') > N.$$
\end{theorem}

%With $N=2k$ we get the following corollary.

%\begin{corollary}\label{cor:statpoint:vec}
%Let $x'$ be a stationary point of $\kfir$, let $z'$ be given by \eqref{y}, and assume that
%\begin{equation}\label{e4c}|z_i'|\not\in\left[{\beta_{2k}^2}{\sqrt{\mu}},\frac{1}{\beta_{2k}^2}{\sqrt{\mu}}\right].\end{equation}
%If $x''$ is another stationary point of $\kfir$ then $\card(x'') > k$.
%\end{corollary}

Note that we allow $\delta_N^-<0$ in the above theorem, in which case the condition on $z'$ is automatically satisfied. The proof depends on a sequence of lemmas, and is given at the end of the section. Clearly, we will rely on Proposition \ref{thm:statpoint}, which requires an investigation of the functional $\G$ \eqref{gdef} and in particular its sub-differential. Introducing the function $g$ as
\begin{equation}
g(x) =
\begin{cases}
\frac{\mu + {|x|^2}}{2} & |x| \geq {\sqrt{\mu}} \\
\sqrt{\mu}|x| & 0 \leq |x| \leq {\sqrt{\mu}}
\end{cases}.
\label{eq:gdef}
\end{equation}
we get \begin{equation}\label{Gg}
\G(x)=\sum_{j=1}^n g(x_j).\end{equation}
Its sub-differential is given by
\begin{equation}
\partial g(x) =
\begin{cases}
\{ x\} & |x| \geq {\sqrt{\mu}} \\
\{\sqrt{\mu}\frac{x}{|x|}\} & 0 < |x| \leq {\sqrt{\mu}} \\
\sqrt{\mu}~\D & x = 0
\end{cases}
\label{eq:vectorsubgrad}
\end{equation} where $\D$ is the closed unit disc in $\C$ or, if working over $\R$, $\D=[-1,1]$. In the remainder we suppose for concreteness that we work over $\C$ (but show the real case in pictures).
Note that the sub-differential consists of a single point for each $x \neq 0$.
Figure~\ref{fig:gfunk} illustrates $g$ and its sub-differential.
\begin{figure}[htb]
\begin{center}
\includegraphics[width=40mm]{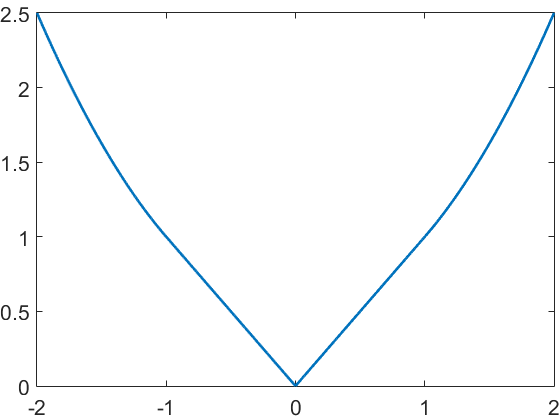}
\includegraphics[width=40mm]{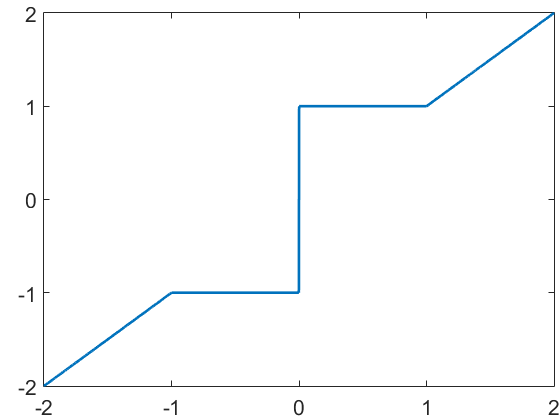}
\end{center}
\caption{The function $g(x)$ (left) and its sub-differential $\partial g(x)$ (right), for $\mu=1$. Note that the sub-differential contains a unique element everywhere except at $x=0$.}
\label{fig:gfunk}
\end{figure}

%Our goal is now to find constraints that assure that this system of equations have only one solution. Before getting into the details we outline the overall idea.
%For simplicity consider two differentiable strictly convex functions $\tilde{h}$ and $\tilde{g}$.
%Their sum is minimized by the stationary point $x_s$ fulfilling $- \nabla \tilde{h}(x_s) = \nabla \tilde{g}(x_s)$.
%Since $\tilde{g}$ is strictly convex its directional derivative $\nabla \tilde{g} (x_s + t \v)^T \v$ is increasing for all $\v \neq 0$.
%Similarly $-\nabla \tilde{h} (x_s + t \v)^T \v$ is decreasing for all $\v\neq 0$ since $-\tilde{h}$ is strictly concave. Therefore $-\nabla \tilde{h} (x_s + t \v)^T \v < \nabla \tilde{g} (x_s + t \v)^T \v$ which means that $x_s$ is the only stationary point.
%In what follows we will estimate the growth of the directional derivatives of the functions involved in \eqref{eq:vectorobj} in order to show a similar contradiction.
%For the function $h$ we do not have convexity, however due to \eqref{eq:vectorRIP} we shall see that it behaves essentially like a convex function for sparse vectors $x$. Additionally we have to consider the non-convex perturbation $-\delta \|x\|^2$. Therefore we need somewhat sharper estimates than just growth of the directional derivatives of $g$.

The following two results establish a bound on the sub-gradients of $\G$.  We begin with some one-dimensional estimates of $g$.
%Due to convexity of $g$ it is clear that at any point $x$ these are growing in all directions $\v$.
%However, in order to guarantee uniqueness of a sparse stationary point we need to show that they grow faster than $2\delta \|\v\|^2$.
%Loosely speaking we will show that provided that the vector $\frac{1}{2}z$, where $z \in \partial g (x)$ is the sub-gradient, has elements that are not too close to the thresholds $\pm \sqrt{\mu}$ this will be true.

\begin{lemma}\label{lemma:bnd1}
Assume that $z_0 \in \partial g(x_0)$ and $\delta_N^->0$.
If
\begin{equation}
\left|
z_0
\right| > \frac{\sqrt{\mu}}{1-\delta_N^-}
\label{eq:subgradbnd1}
\end{equation}
then for any $x_1,~z_1$ with $z_1\in \partial g(x_1)$ and $x_1\neq x_0$, we have
\begin{equation}\label{fd}
\re(z_1 -z_0)\overline{(x_1-x_0)} > \delta_N^- |x_1-x_0|^2 .
\end{equation}
\end{lemma}
\begin{proof}
By rotational symmetry (i.e. $\partial g(e^{i\phi}x)= e^{i\phi} \partial g(x)$), it is no restriction to assume that $z_0>0$. By $\frac{1}{1-\delta_N^-}> 1$ and \eqref{eq:subgradbnd1}, we see that $z_0>\sqrt{\mu}$, and hence the identity $z_0\in\partial g(x_0)$ and \eqref{eq:vectorsubgrad} together imply that $z_0=x_0$ and in particular that $x_0\in\R$ and
\begin{equation}
\label{onion}x_0>\frac{\sqrt{\mu}}{1-\delta_N^-}.
\end{equation}
To prove the result we now minimize the quotient
\begin{equation}\label{gt}\frac{\re(z_1 -z_0)\overline{(x_1-x_0)}}{ |x_1-x_0|^2}\end{equation} and show that it is larger than $\delta_N^-$.
There are three cases to consider; $x_1=0$, $0<|x_1|<\sqrt{\mu}$ and $|x_1|\geq \sqrt{\mu}$. The latter case is easy since then $z_1-z_0=x_1-x_0$ and since $\delta_N^- < 1$, the desired conclusion is immediate.

For the two other cases we first show that $z_1$ and $x_1$ can be assumed to be real.
If $x_1=0$ the above quotient is equivalent to $1-\re(z_1/x_0)$ over $z_1 \in \sqrt{\mu}\D$, since $x_0=z_0$ is real and positive, which is clearly minimized for the real value $z_1 = \sqrt{\mu}$.

For the middle case, $z_1$ and $x_1$ have the same angle with $\R$. We first hold the radii fixed and only consider the angle as an argument. Recall $z_0=x_0$ and set $R=|z_1|/|x_1|$. Then we have
 \[ \begin{split} & \re(z_1 -z_0)\overline{(x_1-x_0)} ={R|x_1|^2-(R+1)\re x_1 \overline{x_0}+|x_0|^2} \\ %& = R|x_1|^2-(R+1)\re x_1 \overline{x_0}+|x_0|^2 \\
 & =\frac{1}{2}\Big((R-1)|x_1|^2+(R+1)|x_1-x_0|^2 + (1-R) |x_0|^2 \Big).  \end{split}  \] So the quotient \eqref{gt} only depends on $|x_1-x_0|^2$ (for fixed radii),
which shows that the quotient is minimized when $x_1$ is real (which then automatically applies to $z_1$ as well).

Summarizing the above we may thus assume that $x_1$ and $z_1$ are real and $x_1\in[-\sqrt{\mu},\sqrt{\mu}]$, which simplifies the quotient \eqref{gt} to $\frac{x_0-z_1}{x_0- x_1}$. We now hold $x_1,~z_1$ fixed and consider $x_0$ as the variable. Recall that $|z_1|\geq |x_1|$. If these are negative we immediately get that the quotient is $\geq 1>\delta_{N}^-$ and the proof is done. In the positive case, the quotient is minimized when $x_0$ is as small as possible (since $z_1\geq x_1$). By \eqref{onion} we hence conclude that the minimum of \eqref{gt} is strictly greater than $\frac{\frac{\sqrt{\mu}}{1-\delta_N^-}-z_1}{ \frac{\sqrt{\mu}}{1-\delta_N^-}-x_1}$. The minimum of this is in its turn clearly attained at $x_1=0$ and $z_1=\sqrt{\mu}$. Summing up, we have that
$$\frac{\re(z_1 -z_0)\overline{(x_1-x_0)}}{ |x_1-x_0|^2}>\frac{\frac{\sqrt{\mu}}{1-\delta_N^-}-z_1}{ \frac{\sqrt{\mu}}{1-\delta_N^-}-x_1}\geq \frac{\frac{\sqrt{\mu}}{1-\delta_N^-}-\sqrt{\mu}}{ \frac{\sqrt{\mu}}{1-\delta_N^-}}=\delta_N^-.$$
\end{proof}

\begin{lemma}\label{lemma:bnd2}

Assume that $z_0 \in \partial g(x_0)$ and $\delta_N^->0$.
If
\begin{equation}
\left|
z_0
\right| < {(1-\delta_N^-)}{\sqrt{\mu}}
\label{eq:subgradbnd2}
\end{equation}
then for any $x_1,~z_1$ with $z_1\in \partial g(x_1)$, $x_1\neq x_0$, we have
\begin{equation}\label{fd1}
\re(z_1 -z_0)\overline{(x_1-x_0)} > \delta_N^- |x_1-x_0|^2.
\end{equation}\end{lemma}
\begin{proof}
The proof is similar to the previous lemma. We first note that $x_0=0$, $x_1\neq 0$ and that $z_0$ may be assumed to be in $(0,(1-\delta_N^-)\sqrt{\mu})$ by rotational symmetry. For a fixed radius $R = \frac{|z_1|}{|x_1|}$ the quotient \[\frac{\re(z_1 -z_0)\overline{(x_1-x_0)}}{ |x_1-x_0|^2}=\frac{\re(z_1 -z_0)\overline{x_1}}{ |x_1|^2} = R- z_0\frac{\re  x_1}{|x_1|^2} \] is smallest when $x_1$ is real valued and positive (which then also applies to $z_1$ which by \eqref{eq:vectorsubgrad} equals $\max(x_1,\sqrt{\mu})$). The expression in question then becomes $R- \frac{z_0}{x_1}$, which is minimized by maximizing $z_0$.
For any choice of $x_1$, \eqref{eq:subgradbnd2} implies that our expression is strictly bigger than
 $$\frac{z_1 -{(1-\delta_N^-)}{\sqrt{\mu}}}{ x_1}=\frac{\max(x_1,\sqrt{\mu}) -{(1-\delta_N^-)}{\sqrt{\mu}}}{ x_1}.$$
 Basic calculus shows that the minimum of this quantity is attained at $x_1 = \sqrt{\mu}$ and equals $\delta_N^-,$ as desired.
\end{proof}

We are now ready to prove Theorem \ref{thm:statpoint:vec}.

\begin{proof}[Proof of Theorem \ref{thm:statpoint:vec}]
By Proposition \ref{thm:statpoint} it suffices to verify \begin{equation}\label{g}
\re\skal{z''-z',x''-x'} > \delta_N^- \|x''-x'\|_2^2,\quad x''\neq x'.
\end{equation} The claim will follow by contradiction.
Suppose first that $\delta_N^->0$. Since $\partial \G(x)=\sum_{j=1}^n \partial g(x_j)$, Lemmas \ref{lemma:bnd1} and \ref{lemma:bnd2} imply that
$$\re(z''_i-z'_i) \overline{(x_i''-x_i')} > \delta_N^- |x_i''-x_i'|^2,
$$
for all $i$ with $x_i''-x_i' \neq 0$. Since $x_i''-x_i'=0$ gives $(z''_i-z'_i)\overline{(x_i''-x'_i)}=0$
summing over $i$ gives the result.

Suppose now that $\delta_N^-<0$. By \eqref{g} it suffices to prove that $\re\skal{z''-z',x''-x'}\geq 0$ for all $x''\neq x'$. Fix $i$ in $\{1,\ldots,n\}$. By rotational symmetry it is easy to see that we can assume that $x'_i,z'_i\geq 0$. Moreover, for fixed values of $|z_i''|$ and $|x''_i|$ (but variable complex phase) it is easy to see that $\re (z_i''-z'_i)(x_i''-x_i')$ achieves min when these are also real, i.e. we can assume that $x''_i,z''_i\in\R$. Since the graph of $\partial g$ is non-decreasing it follows that $(z_i''-z_i')(x_i''-x_i') \geq 0$ for all $i$, as desired.

It remains to consider the case when $\delta_N^-=0$, and as above we reach a contradiction if we prove that $\re\skal{z''-z',x''-x'}> 0$. Again we can assume that $x'_i,z'_i\geq 0$ and that $x''_i,z''_i\in\R$. Then \eqref{e4} implies that $z_i'\neq \sqrt{\mu}$ for all $1\leq i \leq n$, which via $z'_i\in \partial g(x'_i)$ also implies that $x_i'\not\in (0,\sqrt{\mu}]$. If $x''\neq x'$ we must have $x''_i\neq x_i'$ for some $i$. Using that $z_i''\in \partial g(x_i'')$, examination of \eqref{eq:vectorsubgrad} yields that also $z_i''\neq z_i'$.
With this at hand we see that the left hand side of \eqref{g} is strictly positive, whereas the right equals 0, which again is a contradiction.

\end{proof}

\iffalse
For later use in [..], we record that we have shown the following result. \begin{corollary}\label{corbor}
Let $z'$ satisfy $z'\in\partial \G(x')$ and
\begin{equation}\label{e4t}|z_i'|\not\in\left[{\lambda^2}{\sqrt{\mu}},\frac{1}{\lambda^2}{\sqrt{\mu}}\right].\end{equation}
If $z''\in\partial \G(x'')$, then $\re\scal{z''-z',x''-x'}>(1-\lambda^2)\|x''-x'\|_2^2$.
\end{corollary}
\fi

\subsection{Conditions on global minimality}\label{erai1}

\begin{theorem}\label{thm:globalpoint2}
Let $A$ satisfy $\|A\|_{\infty,col}\leq 1$, let $x'$ be a stationary point of $\kfir$ and let $z'$ be given by \eqref{y}. Assume that
\begin{equation}\label{cond3}|z_i'|\not\in\left[{(1-\delta_N^-)}{\sqrt{\mu}},\frac{1}{1-\delta_N^-}{\sqrt{\mu}}\right],\quad 1\leq i\leq n.\end{equation}
If \begin{equation}\label{cond2}2\mu\card(x') +\fro{Ax'-b}< \mu N+\mu,\end{equation} then $x'$ is the unique global minimum of $\kfi$ and $\kfir$.
\end{theorem}

Obviously, it is desirable to pick $N$ as large as possible, which is limited by \eqref{cond3} and the fact that $\delta_N^-$ increases with $N$. Also note that $\delta_N^-\geq 0$ since $1-\delta_1^-\leq \min_i\{\|a_i\|_2^2\}=\|A\|_{\infty,col}^2\leq 1$ so $\delta_1^-\geq 0$.

\begin{proof}
Set $k=\card(x')$ and assume that $x'$ is not the unique global minimizer of $\kfir$. Let $x''$ be another. Either $\kfir(x)=\kfi(x)$ for $x=x''$ or $x''$ is part of a connected component of global minimizers to $\kfir$ including two points that satisfy the equation, by Theorem \ref{celok5}. Since one of these must be different from $x'$, we may assume that $\kfir(x'')=\kfi(x'')$.
Theorem \ref{thm:statpoint:vec} then shows that $\card (x'')\geq N-k+1$.
Since $\kfi(x'')=\kfir(x'')$ and $\kfi(x')\geq\kfir(x')$ it follows from \eqref{cond2} that
$$\kfir(x'')-\kfir(x')\geq \kfi(x'')-\kfi(x')\geq \mu(N-k+1)-(\mu k +\fro{Ax'-b})>0.$$
This is a contradiction, and hence $x'$ must be the unique global minimizer of $\kfir$. By Theorem \ref{celok5} it then follows that $x'$ is also unique minimizer of $\kfi$.
\end{proof}

\subsection{Finding the oracle solution}
In this final subsection we return to the compressed sensing problem of retrieving a sparse vector $x_0$ given corrupted measurements $b=Ax_0+\ep$, where $\ep$ is noise and $x_0$ is sparse. More precisely we set $S=\supp x_0$ where we assume that $\#S=k$ is much smaller than $m$ -- the amount of rows in $A$ (i.e. number of measurements). Here $\#S$ denotes the amount of elements in $S$ and the noise can be of any type, our theory only relies on knowledge of $\|\ep\|$.

We let $x_{0,j}$ denote the elements of the vector $x_0$. Let $A_S$ denote the matrix obtained from $A$ by setting columns outside of $S$ to $0$, and let $x_{or}$ denote the least squares solution to $A_S x_{or}=b$. Note that this is the so called ``oracle solution'' discussed in the introduction, which can also be written $x_{or}=(A_S^*A_S)^\dagger A_S^*b$ where $(A_S^*A_S)^\dagger$ denotes the Moore-Penrose inverse.

Our first result collects some general observations about the oracle solution.

\begin{proposition}\label{p11}
Let $A$ satisfy $\|A\|_{\infty,col}\leq 1$ and let $c>0$. If $$|x_{0,j}|>c+\frac{\|\ep\|_2}{\sqrt{1-\delta_k^-}}$$ for all $j\in S$ then the oracle solution $x'=x_{or}$ satisfies $\supp (x')=\supp (x_0)$. We also have $|x_{j}'|>c,~ j\in S,$ $\|Ax'-b\|_2\leq \|\epsilon\|_2$, and $$\|x'-x_0\|_2\leq \frac{\|\ep\|_2}{\sqrt{1-\delta_k^-}}.$$
\end{proposition}
\begin{proof}
Consider the equation $A_Sx=Ax_0+\ep$ and note that $Ax_0=A_Sx_0$. The least squares solution is obtained by applying $(A_S^*A_S)^\dagger A_S^*$ which gives the solution $$x'=x_0+(A_S^*A_S)^\dagger A_S^*\ep=x_0+\eta,$$ where we set $(A_S^*A_S)^\dagger A_S^*\ep=\eta$. By construction of the Moore-Penrose inverse, $\supp \eta\subset S$, and hence $$A\eta=A_S\eta=P_{\text{Ran} A_S}\ep,$$ where $P_{\text{Ran} A_S}$ denotes the orthogonal projection onto the range of $A_S.$ In particular, $$\|\eta \|_2\leq\frac{\|A_S\eta\|_2}{\sqrt{1-\delta_k^-}}=\frac{\|P_{\text{Ran} A_S}\ep\|_2}{\sqrt{1-\delta_k^-}}\leq \frac{\|\ep\|_2}{\sqrt{1-\delta_k^-}},$$
which establishes the final inequality in the proposition. Also $\|\eta\|_{\infty}\leq \|\eta\|_{2}$ which implies \begin{equation}\label{sofg}|x_{j}'|\geq |x_{0,j}|-|\eta_j|>c+\frac{\|\ep\|_2}{\sqrt{1-\delta_k^-}}-\frac{\|\ep\|_2}{\sqrt{1-\delta_k^-}}=c,\quad j\in S.\end{equation} This also gives $\supp x'=\supp x_0$ since by construction we clearly have \[\supp x' \subset \supp x_0\cup\supp \eta \subset S.\]
Finally, consider $Ax'-b$, which equals \begin{equation}\label{sof}\begin{aligned}&Ax'-b=A_Sx'-b=\\&A_Sx_0+A_S(A_S^*A_S)^\dagger A_S^*\ep-(A_Sx_0+\ep)=(P_{\text{Ran} A_S}-I)\ep=-P_{(\text{Ran} A_{S})^{\perp}}\ep\end{aligned}\end{equation}
and hence $\|Ax'-b\|_2\leq \|\epsilon\|_2$.
\end{proof}
The below proposition shows that the oracle solution is under mild assumptions a local minimizer of $\kfir$, which we denote by $x'$ for notational consistency.

\begin{proposition}\label{p1}
Let $A$ satisfy $\|A\|_{\infty,col}\leq 1$. If $\|\ep\|_2< {\sqrt{\mu}}$ and $$|x_{0,j}|>\sqrt{\mu}+\frac{\|\ep\|_2}{\sqrt{1-\delta_k^-}}$$ for all $j\in S$ then the oracle solution $x'=x_{or}$ is a strict local minimum to $\kfir$ with $\supp (x')=\supp (x_0)$. We also have $|x_{j}'|>\sqrt{\mu},~ j\in S,$ $\|Ax'-b\|_2\leq \|\epsilon\|_2$, and $$\|x'-x_0\|_2\leq \frac{\|\ep\|_2}{\sqrt{1-\delta_k^-}}.$$
\end{proposition}
%do we need to prove this or showing stationary point is enough??????????
\begin{proof}
All inequalities follow by applying Proposition \ref{p11} with $c=\sqrt{\mu}$, so it remains to prove that $x'$ is a local minimum of $\kfir=\Q_2(\mu\card)+\fro{Ax-b}$. To this end, consider $\kfir(x'+v)$. Since $|x_{j}'|>\sqrt{\mu}$ for $j\in S$, the term $\Q_2(\mu\card)$ (see \eqref{l0000}) is constant for the corresponding indices of $v$, as long as $v$ is small. For $v$ in a neighborhood of 0 we get \begin{equation*}\kfi(x'+v)=\sum_{j\in S^c}\left(2\sqrt{\mu} |v_j|-|v_j|^2\right)+2\re\scal{v,A^*(Ax'-b)}+\|Av\|_2^2+\kfir(x').\end{equation*} Since $x'$ solves the least squares problem posed initially, the vector $A_S^*(Ax'-b)=A_S^*(A_Sx'-b)$ must be 0. With this in mind the above expression simplifies to \begin{equation}\label{rd1}2\left(\sum_{j\in S^c}\sqrt{\mu} |v_j|+\re\left( v_j \scal{a_j,Ax'-b}\right)\right)-\sum_{j\in S^c}|v_j|^2+\|Av\|_2^2+\kfir(x').\end{equation} By the Cauchy-Schwartz inequality and \eqref{sof} we have $$|\scal{a_j,Ax'-b}|\leq \|a_j\|_2\|\ep\|_2< \|A\|_{\infty,col} {\sqrt{\mu}}\leq\sqrt{\mu}.$$ It follows that the term $\sum_{j\in S^c}\sqrt{\mu} |v_j|+\re \left(v_j\scal{a_j,Ax'-b}\right)$ in \eqref{rd1} can be estimated from below by \[  \sum_{j \in S^c}  |v_j|( \underbrace{\sqrt{\mu} -|\langle a_j, A x' - b \rangle|}_{:= \alpha_j}   ) \ge \alpha \sum_{j\in S^c}|v_j|     \] where \( \alpha=\min_j\{\alpha_j\} > 0  \) for all \(j\). Hence \begin{equation}\label{pki}2\left(\sum_{j\in S^c}\sqrt{\mu} |v_j|+\re \left(v_j\scal{a_j,Ax'-b}\right)\right)-\sum_{j\in S^c}|v_j|^2>0\end{equation} for $v$ in a neighborhood of 0, as long as $\sum_{j\in S^c}|v_j|^2\neq 0$. To have $\kfir(x'+v)\leq\kfir(x'),$ \eqref{rd1} shows that we need the terms in \eqref{pki} to be zero, or equivalently $\supp v\subset S$. But then \eqref{rd1} reduces to $\|Av\|_2^2+\kfir(x')$, and since ${\delta_k^-}<1$ it follows that $\|Av\|_2^2>0$ unless $v=0$. In other words, $x'$ is a strict local minimizer.
\end{proof}

%By the above proof, an explicit formula for $x_1$ is
%\begin{equation}\label{ko}x_1=x_0+(A_S^*A_S)^\dagger A_S^*\ep.\end{equation}

In the above proposition, there is nothing said as to whether $x'$ is a global minimum or not. To get further, let $z'$ correspond to $x'$ via \eqref{y}. We need conditions such that \eqref{cond3} holds for $z'$, i.e.
\begin{equation}\label{cond31}|z_{i}'|\not\in\left[{(1-\delta_N^-)}{\sqrt{\mu}},\frac{{\sqrt{\mu}}}{1-\delta_N^-}\right].\end{equation}
We remind the reader that $N$ is a number which preferably is a bit larger than $2k$, where $k$ is the cardinality of $x_0$.

\begin{proposition}\label{p2}
Let $A$ satisfy $\|A\|_{\infty,col}\leq 1$. If $\|\ep\|_2< {(1-\delta_N^-)\sqrt{\mu}}$ and \begin{equation}\label{po7}|x_{0,j}|>\frac{\sqrt{\mu}}{1-\delta_N^-}+\frac{(1-\delta_N^-)\sqrt{\mu}}{\sqrt{1-\delta_k^-}}, \quad j\in S,\end{equation}
then \eqref{cond31} holds.
\end{proposition}
\begin{proof}
Using \eqref{sof} we get
\begin{align}\label{tg}z'=( I-A^* A)x' + A^* b =  x'-A^*(Ax'-b)= x'+A^*P_{(\text{Ran} A_{S})^{\perp}}\ep.\end{align}
Since $A^*P_{(\text{Ran} A_{S})^{\perp}}$ is 0 on rows with index $j\in S$ (being a scalar product of a vector in $\text{Ran} A_{S}$ and another in its orthogonal complement), we see that $z_{j}'= x_{j}'$ for such $j$. Combining this with the final estimate of Proposition \ref{p11}, we see that  \begin{equation*}\label{po9}|z_{j}'|\geq |x_{0,j}|-|x_{0,j}-x'_j|>\frac{1}{1-\delta_N^-}{\sqrt{\mu}}, \quad j\in S\end{equation*} holds as a consequence of \eqref{po7}. For the remaining $z_{j}'$, (i.e. $j\in S^c$), we have $x_{j}'=0$ so \eqref{tg} implies \begin{equation}\label{o0}\begin{aligned}&|z_j'|=|(A^*P_{(\text{Ran} A_{S})^{\perp}}\ep)_j|=|\scal{P_{(\text{Ran} A_{S})^{\perp}}\ep,a_j}|\leq \\&\|A\|_{\infty,col}\|\ep\|_2\leq \|\ep\|_2<{(1-\delta_N^-)}{\sqrt{\mu}},\end{aligned}\end{equation}
which establishes \eqref{cond31}.
\end{proof}

Putting all the results together and combining with simple estimates, we finally get

\begin{theorem}\label{thm:doctorgadget}
Suppose that $b=Ax_0+\ep$ where $A$ is an $m\times n$-matrix with $\|A\|_{\infty,col}\leq 1$ and set $\card (x_0)=k.$ Let $N\geq 2k$ and assume that
$\|\ep\|_2< {(1-\delta_N^-)\sqrt{\mu}}$ and $$|x_{0,j}|>\para{\frac{1}{1-\delta_N^-}+1}\sqrt{\mu},\quad j\in \supp x_0.$$ Then the oracle solution $x'=x_{or}$ is a unique global minimum to $\kfir$ as well as $\kfi$, with the property that $\supp x'=\supp x_0$, that $$\|x'-x_0\|_2\leq \frac{\|\ep\|_2}{\sqrt{1-\delta_k^-}},$$ and that $\card(x'')> N-k$ for any other stationary point $x''$ of $\kfir$.
\end{theorem}

\begin{proof}
All the statements follow by Theorem \ref{thm:statpoint:vec}, Theorem \ref{thm:globalpoint2} and Proposition \ref{p1}, so we just need to check that these apply. Note that $\sqrt{1-\delta_{N}^-}\leq \sqrt{1-\delta_k^-}\leq \|A\|_{\infty,col}\leq 1$ which will be used repeatedly.

We begin to verify that Proposition \ref{p1} applies, which is easy by noting that $\|\ep\|_2\leq (1-\delta_N^-)\sqrt{\mu}<\sqrt{\mu}$ and $${\sqrt{\mu}}+\frac{\|\epsilon\|_2}{\sqrt{1-\delta_k^-}}\leq \frac{\sqrt{\mu}}{1-\delta_N^-}+\frac{(1-\delta_N^-)\sqrt{\mu}}{\sqrt{1-\delta_k^-}}\leq \frac{\sqrt{\mu}}{1-\delta_N^-}+{\sqrt{\mu}}<|x_{0,j}|.$$

Now, to verify that Theorem \ref{thm:statpoint:vec} applies we need to check the condition \eqref{cond31}, which follows if we show that Proposition \ref{p2} applies. This is almost immediate since the estimate on $\|\ep\|_2$ is satisfied by assumption and \eqref{po7} follows by noting that $\frac{1-\delta_N^-}{\sqrt{1-\delta_k^-}}\leq 1.$ By this we also get the first condition of Theorem \ref{thm:globalpoint2} for free. We are done once we also verify \eqref{cond2}. To this end, note that $\|Ax'-b\|_2\leq \|\ep\|_2< (1-\delta_N^-)\sqrt{\mu}$ by Proposition \ref{p1}, so \eqref{cond2} holds if $2\mu k+(1-\delta_N^-)^2{\mu}\leq \mu N+\mu$, which is clearly the case since $N\geq 2k$.

\end{proof}

As a final remark, a simpler statement is found by setting $N=2k$, which gives the loosest conditions to verify. We spelled this out in Corollary \ref{cor:doctorgadget}, where we also simplified further by replacing $\frac{1}{1-\delta_N^-}+1$ by $\frac{2}{1-\delta_N^-}$, for aesthetic reasons.

\section{Known model order; the $k$-sparsity problem}\label{seccardfix}

Let $P_k=\{x:\card (x)\leq k\}$ where $x$ is a vector in $\C^n$ or $\R^n$. Set $f(x)=\iota_{P_k}(x)$ and note that the problem \begin{equation}\label{agt1}\argmin_{\card(x)\leq k}\|Ax-b\|_2\end{equation} is equivalent to finding the minimum of \begin{equation}\label{agt2}\ksi(x)=\iota_{P_k}(x)+\|Ax-b\|_2^2,\end{equation}
(where we put a subindex $k$ to distinguish from $\kfi$ in the previous section).\footnote{Admittedly, the notation is not perfect since if $k$ and $\mu$ equal the same integer, then the two symbols become the same, but we hope the reader can live with this.} Again, we will approach this problem by using
$$\ksir(x)=\Q_2(\iota_{P_k})(x)+\|Ax-b\|_2^2.$$
This is in some ways much simpler than the situation in the previous sections, for example all local minimizers of $\ksi$ are clearly in $P_k$.
On the other hand, $\Q_2(\iota_{P_k})$ turns out to be rather complicated. We recapitulate the essentials, which follows by adapting the computations in \cite{andersson-etal-ol-2017} (for matrices) to the vector setting. Define $\tilde x$ to be the vector $x$ resorted so that $(|\tilde x_j|)_{j=1}^d$ is a decreasing sequence. Then
%\begin{equation}\label{lust2}\S(\iota_{P_k})(y)= \sum_{j=k+1}^d -|\tilde y_j|^2\end{equation}
\begin{equation}\label{lust1}\Q_2(\iota_{P_k})(x)= \frac{1}{k_*}\para{\sum_{j>k-k_*}|\tilde{x}_j|}^2-\sum_{j>k-k_*}|\tilde x_j|^2\end{equation} where $k_*$ is the largest value of \(l \in \{1,\dots,k\} \) for which the non-increasing sequence
\begin{equation}\label{lust}s(l)=\left(\sum_{j>k-l}|\tilde x_j|\right)-l|\tilde x_{k+1-l}|\end{equation}
is non-negative (note that it clearly is non-negative for $l=1$). For any given vector $x$ this is clearly computable, although one has to go through a number of cases, but the good thing is that there is an efficient way to implement the corresponding proximal operator (discussed in Section \ref{int tch}) so in practice this is of little importance. Although it is not very clear from the above expression, $\Q_2(\iota_{P_k})$ is known to be continuous (see e.g.~Proposition 3.2 in \cite{carlsson2018convex}), and this will be used without comment below. We first show that the global minima of $\ksir$ and $\ksi$ coincide.

\subsection{Equality of minimizers for $\ksi$ and $\ksir$}\label{Kfeas}

As before $A$ is a matrix of size $m\times n$, which we need to impose some additional conditions on. The theory in the entire Section \ref{seccardfix} assumes that
\begin{itemize}
\item [(A1)] $n\geq m+k+2$ (when working over the reals) whereas $n\geq 2m+k+2$ when working in $\C^n$.
\item [(A2)] Either $\|A\|_{\infty,col}< 1$ \textit{or}\newline $\|A\|_{\infty,col}\leq 1$ and all possible scalar products $\scal{a_i,a_j}$ are non-zero.
\end{itemize}
The equivalent of Theorem \ref{celok5} now reads.

\begin{theorem}\label{celok6}
Under assumption (A1-A2) all local minimzers of $\ksir$ lie in $P_k$ (and hence are minimizers to $\ksi$). In particular the global minimizers exist and coincide.
\end{theorem}

We note that the conclusion is the same as that of Theorem 5.1 in \cite{carlsson2018convex}, which holds for almost any penalty $f$. However, that proof assumes that $\|A\|<1$ which is unnecessarily strong in the present setting. For example it would rule out all Gaussian random matrices with normalized columns. The proof of Theorem \ref{celok6} is given in Appendix \ref{kfesappendix}.

\subsection{On the uniqueness of sparse stationary points}

We now give a condition, similar to \eqref{e4} in Section \ref{erai}, to ensure that a sparse stationary point is unique, in the sense that other stationary points must have higher cardinality.

\begin{theorem}\label{thm:statpoint:vecfix}
Let $x'$ be a stationary point of $\ksir$ with cardinality $k$, let $z'$ be given by \eqref{y}, and assume that
\begin{equation}\label{e4fix}|\tilde z'_{k+1}|<(1-2\delta_{2k}^-)|\tilde z'_{k}|.\end{equation}
If $x''$ is another stationary point of $\ksir$ then $\card(x'') > k$.
\end{theorem}

Again, we allow $\delta_{2k}^-<0$ in the above theorem, in which case the condition on $z$ is automatically satisfied. We begin with a lemma. Recall $\G$ given by \eqref{gdef}, i.e.~$\frac{1}{2}\Q_2(\iota_{P_k})(x)+\frac{1}{2}\fro{x}$ in the present case. We need an expression for $\partial\G(x)$ for $x\in P_k$.

\begin{lemma}\label{sun}
If $x\in P_k$ then $z\in\partial\G(x)$ if and only if $ z_j= x_j$ for $j \in \supp x$ and $ z_j\in |\tilde x_k|\D $ for all other $j$.
\end{lemma}
\begin{proof}
Since $\Q_2(\iota_{P_k})+\fro{x}$ is the l.s.c. convex envelope of $\iota_{P_k}+\fro{x}$, we have that $\G(x)=\frac{1}{2}\Q_2(\iota_{P_k})+\frac{1}{2}\fro{x}$ is the double Fenchel conjugate of $\frac{1}{2}\iota_{P_k}+\frac{1}{2}\fro{x}$. The Fenchel conjugate of the latter is easily computed to $$\G^*(y)=\frac{1}{2}\sum_{j=1}^k |\tilde y_j|^2.$$ By the well-known identity $z\in\partial \G(x)\Leftrightarrow x\in\partial \G^*(z)$ (see e.g. Proposition 16.9 in \cite{bauschke2017convex}) we have $z\in\partial\G(x)$ if and only if
$$
\G^*(w) \geq \G^*(z) + \scal{x,w-z},
$$
for all $w$ which means that
\begin{equation}\label{t6}z=\argmax_z \re\scal{x,z}-\frac{1}{2}\sum_{j=1}^k |\tilde z_j|^2.\end{equation}
By standard results on reordering of sequences (see e.g.~Ch.~1 in \cite{simon2005trace}), the maximum is attained for a $z$ which is ordered in the same way as $x$. In other words we can choose a permutation $\pi$ such that $|x(\pi(j))|=|\tilde x_j|$ and $|z(\pi(j))|=|\tilde z_j|$ holds for all $j$. This in turn implies that $\scal{x,z}=\sum_{j=1}^n x(\pi(j))\overline{z(\pi(j))}$. Combined with ${x}(\pi(j))=0$ for $j>k$, we see that \eqref{t6} turns into \begin{equation}\label{t7}z=\frac{1}{2}\argmax_z-\sum_{j=1}^k | x_j(\pi(j))- z_j(\pi(j))|^2.\end{equation} The lemma now easily follows.
\end{proof}

\begin{proof}[Proof of Theorem \ref{thm:statpoint:vecfix}.]
If $\card (x'')\leq k$ we clearly have $x''-x'\in P_{2k}$ and both $z'$ and $z''$ have the structure stipulated in Lemma \ref{sun}. Let $I'=\supp x'$ and $I''=\supp x''$. Then $\re\scal{z''-z',x''-x'}$ can be written
\begin{equation}
\re\left(\sum_{\footnotesize
\begin{array}{c}
i \in I' \\
i \in I''\end{array}}
|x_i''-x_i'|^2 +
\sum_{\footnotesize
\begin{array}{c}
i \in I' \\
i \notin I''\end{array}}
(x_i'-z_i'')\overline{x_i'} +
\sum_{\footnotesize
\begin{array}{c}
i \notin I' \\
i \in I''\end{array}}
(x_i''-z_i')\overline{x_i''}\right).
\label{eq:sum1}
\end{equation}
As before we want to reach a contradiction to Proposition \ref{thm:statpoint}, i.e.~we want to prove $\re\scal{z''-z',x''-x'}>\delta_{2k}^-\|x''-x'\|_2^2$. Note that
\begin{equation}\label{gr2}
\|x''-x'\|_2^2=\sum_{\footnotesize
\begin{array}{c}
i \in I' \\
i \in I''\end{array}}
|x_i''-x_i'|^2 +
\sum_{\footnotesize
\begin{array}{c}
i \in I' \\
i \notin I''\end{array}}
|x_i'|^2 +
\sum_{\footnotesize
\begin{array}{c}
i \notin I' \\
i \in I''\end{array}}
|x_{i}''|^{2},
\end{equation}
that the first term in \eqref{eq:sum1} and \eqref{gr2} are the same, and that $\delta_{2k}^-<1$. Since the second and third sums have the same number of terms it suffices to show that
\begin{equation}
\re (x_i'-z_i'')\overline{x_i'} + (x_j''-z_j')\overline{x_j''} > \delta_{2k}^-(|x_i'|^2 + |x_j''|^{ 2}),
\label{eq:twoterms}
\end{equation}
for any pair $i\in I'$, $i\notin I''$ and $j \notin I'$, $j \in I''$. This in turn will follow upon showing that
$$z_i''\overline{x_i'}+z_j'\overline{x_j''} \leq |z_i''| |x_i'| + |z_j'| |x_j''| < (1-\delta_{2k}^-)(|x_i'|^2 + |x_j''|^{ 2}).$$
Since $i\notin I''$ and $j\in I''$ we have $|z_i''|\leq |z_j''|$ by Lemma \ref{sun}, as well as that $z_j''=x_j''$. Turning to $z_j'$ we can say more due to assumption \eqref{e4fix}. More precisely, since $i\in I'$ and $j\notin I'$ we have $|z_j'|< (1-2\delta_{2k}^-)| z'_{i}|=(1-2\delta_{2k}^-)| x'_{i}|$, where again Lemma \ref{sun} was used in the last identity. Summing up we have \begin{equation*}\begin{split}|z_i''| |x_i'| + |z_j'| |x_j''|  <  |x_j''||x_i'| + (1-2\delta_{2k}^-)| x'_{i}| |x_j''| & = 2(1-\delta_{2k}^-)| x'_{i}| |x_j''| \\ & \leq (1-\delta_{2k}^-)(|x_i'|^2 + |x_j''|^{ 2}),\end{split}\end{equation*}
as desired.
\end{proof}

\subsection{Conditions on global minimality}
The statements in this section are actually quite a bit stronger than the corresponding ones in Section  \ref{erai1}. On the other hand, the condition \eqref{e4fix} entails that we must have $\delta_{2k}^-<1/2$, which limits the applicability.

\begin{theorem}\label{thm:globalpoint2f}
Let $A$ satisfy (A1-A2) and let $x'\in P_k$ be a stationary point of $\ksir$. Let $z'$ be given by \eqref{y} and assume that \eqref{e4fix} applies. Then $x'$ is a unique global minimizer of $\ksi$ and $\ksir$, and $\ksir$ has no other local minimizers either.
\end{theorem}
\begin{proof}
By Theorem \ref{celok6} there exists $x''\in P_k$ which is a global minimizer for both $\ksi$ and $\ksir$. Clearly $x''$ is then a stationary point, so if $x'\neq x''$ this would contradict Theorem \ref{thm:statpoint:vecfix}, so we must have $x'=x''$. The same argument works for the local minimizers.
\end{proof}

\subsection{Finding the oracle solution}
We now assume that $b$ is of the form $Ax_0+\ep$ where $\ep$ is noise and $x_0$ is sparse. More precisely we set $S=\supp x_0$ where we assume that $\#S=k$.
As before let $A_S$ denote the matrix obtained from $A$ by setting columns outside of $S$ to $0$. %Applying Proposition \ref{p11} with $c=\|\ep\|_2$ we obtain.

%\begin{proposition}\label{p1f}
%Let $A$ satisfy $\|A\|_{\infty,col}\leq 1$. If $$|x_{0,j}|>\para{1+\frac{1}{\sqrt{1-\delta_k^-}}}\|\ep\|_2\neq 0$$ for all $j\in S$ then the oracle solution $x'=x_{or}$ satisfies $\supp (x')=\supp (x_0)$, $|x'_j|>\|\epsilon\|_2$ for all $j\in S$, $\|Ax'-b\|_2\leq \|\epsilon\|_2$ and $$\|x'-x_0\|_2\leq \frac{\|\ep\|_2}{\sqrt{1-\delta_k^-}}.$$
%\end{proposition}

In this case, Theorem \ref{celok6} is strong enough so that we do not need any longer argument to establish that $x_{or}$ is a global minimizer, since all local minimizers of $\ksir$ are to be found in $P_k$. We obtain the following result.

\begin{theorem}\label{cor:dogadgetf}
Let $A$ satisfy (A1-A2). If $\ep\neq 0$ and $$\min_{j\in S}|x_{0,j}|>\para{\frac{\|\ep\|_2}{\sqrt{1-\delta_{k}^-}}+\frac{2\|\ep\|_2}{\sqrt{1-\delta_{2k}^-}}},$$ then the estimates of Proposition \ref{p11} applies and the oracle solution is a global minimum of $\ksi$ and $\ksir$.
\end{theorem}

\begin{proof}
Proposition \ref{p11} immediately applies with $c=\frac{2\|\ep\|_2}{\sqrt{1-\delta_k^-}}$.
Let $J\subset \{1,\ldots,n\}$ have cardinality $k$ and consider the problem $$x_J=\argmin_x \|A_Jx-b\|^2.$$ Searching over $J$ gives rise to (at most) $\binom{n}{k}$ points (since $\delta_k^-<1$), among which the minimizers of $\ksi$ are found (see Lemma \ref{l4} for more details). By Theorem \ref{celok6} a subset of these are the local minimizers of $\ksir$, and the global minimizer must be the one that gives the lowest value for $\|A_Jx-b\|^2$. With this notation we have $x_{or}=x_S$ and the estimates of Proposition \ref{p11} gives $\|Ax_S-b\|_2\leq \|\ep\|_2$ and $$|x_{S,j}|>{\frac{2\|\ep\|_2}{\sqrt{1-\delta_k^-}}},\quad j\in S.$$
If $J\neq S$ is another set with cardinality $k$ then $x_J$ and $x_S$ must differ in at least one coordinate, so
$$\|x_{S}-x_J\|_2>{\frac{2\|\ep\|_2}{\sqrt{1-\delta_{2k}^-}}},\quad j\in S.$$
But then $$\|Ax_J-b\|_2=\|A(x_J-x_S)+Ax_S-b\|_2\geq\sqrt{1-\delta_{2k}^-}\|x_S-x_J\|_2-\|\ep\|_2>\|\ep\|_2.$$
Thus $x_S$ is the one with the lowest value for $\|Ax_J-b\|_2$, which was to be shown.
\end{proof}

When minimizing ${\ksir}$ in practice, it would of course be good to know if there are local minima where one can get stuck. To rule out this possibility, we need unfortunately to assume that $\delta_{2k}^-<1/2$.

\begin{corollary}\label{new}
If in addition to what is assumed in Theorem \ref{cor:dogadgetf} we have $$\min_{j\in S}|x_{0,j}|>\para{\frac{1}{1-2\delta_{2k}^-}+\frac{1}{\sqrt{1-\delta_k^-}}}\|\ep\|_2, $$ then the there are no local minimizers of $\ksir$ except the oracle solution.
\end{corollary}
\begin{proof}
Theorem \ref{cor:dogadgetf} clearly ensures that $x'=x_{or}$ is a stationary point. The desired result follows from Theorem \ref{thm:globalpoint2f} once we verify that \eqref{e4fix} applies for $z'$ given by \eqref{y}. We need to check that $|\tilde z'_{k+1}|<(1-2\delta_{2k}^-)|\tilde z'_k|$. Note that $|\tilde z'_{k+1}|\leq\|\ep\|_2$ by the same estimate as \eqref{o0}. Moreover, since $z'\in\partial \G(x')$, Lemma \ref{sun} implies that $|\tilde z'_k|=|\tilde x'_k|$ so it suffices to show that $\|\ep\|_2<(1-2\delta_{2k}^-)|\tilde x'_k|$. This in turn holds by applying Proposition \ref{p11} with $c=\frac{\|\ep\|^2}{1-2\delta_{2k}^-}$, and the proof is complete.

\end{proof}

 \section{Experimental Evaluation}\label{sec:num}

In this section we present an experiment designed to validate our main theoretical results.
Our goal is to verify that both the proposed methods are able to recover the oracle solution when the signal to noise level is sufficiently large.
For both our formulation we need to specify a parameter; $\mu$ in case of $\K_{\mu,reg}$ and $k$ for $\K_{k,reg}$.
Since we are working with synthetic data generated by $b = A x_0 + \epsilon$, with a known vector $x_0$ we can set $\mu$ so that the non-zero elements are large enough to be preserved, and $k$ so that $k=\supp(x_0)$
(see Section~\ref{sec:num2} for a more detailed description).

In realistic settings where $x_0$ is unknown selecting parameters is more difficult and requires a precise definition of what constitutes a good solution. This could be based on application specific prior information about the support or size of the elements. If the size of the correct support is assumed to be known, then $\K_{k,reg}$ is the convenient choice. On the other hand formulations able to directly specify the sought cardinality are uncommon. Therefore soft penalties such as $\lambda\|\cdot\|_1$ or $\Q_{2}(\mu\card)$ are often utilized by searching over the parameter until a suitable cardinality solution is found.
The $\ell_1$-norm has been used in this way for a number of practical applications e.g. face recognition \cite{wright-etal-pami-2009}, subspace clustering \cite{elhamifar-vidal-pami-2013}, non rigid structure from motion \cite{kong-lucey-cvpr-2016} and outlier detection \cite{olsson-etal-cvpr-2010}, diffraction imaging \cite{SHI2020107350}, MRI tomography \cite{7361050} to name a few.

In \cite{carlsson2019unbiased,carlsson2020perfect} solutions of a given cardinality was recovered using $\Q_2(\mu\card)$ by searching over $\mu$. Note however that while it is 1-dimensional, the search criterion is not guaranteed to be unimodal and it is not clear over what range nor at what density one needs to sample in order not to miss the sought solution.

\subsection{Numerical Recovery Results}\label{sec:num2}

In \cite{candes2008enhancing} astonishing results are shown in the noise free case. For example in Figure 2 (of that paper) we see how $k=130$ non-zero entries are recovered using a matrix $A$ of size $m\times n= 256\times 512$, (which incidentally is close to the theoretical bound $2k\leq m$ in the present paper\footnote{Note indeed that the condition $\delta_{2k}^-<1$ is equivalent to any $2k$ columns of $A$ being linearly independent, which holds with probability 1 for Gaussian random matrices as long as $2k\leq m$.}). However, in the presence of noise, performance seems to drop drastically. In Figure 7 (of the same paper) we see an example where performance is evaluated with $k=8$, $m=72$ and $n=256$. %This is in line with the predictions of \cite{loh2017support}, which use $k=\sqrt{m}$ in their numerical section 4.3.

\begin{figure}[htb]
     \includegraphics[width=.5\textwidth]{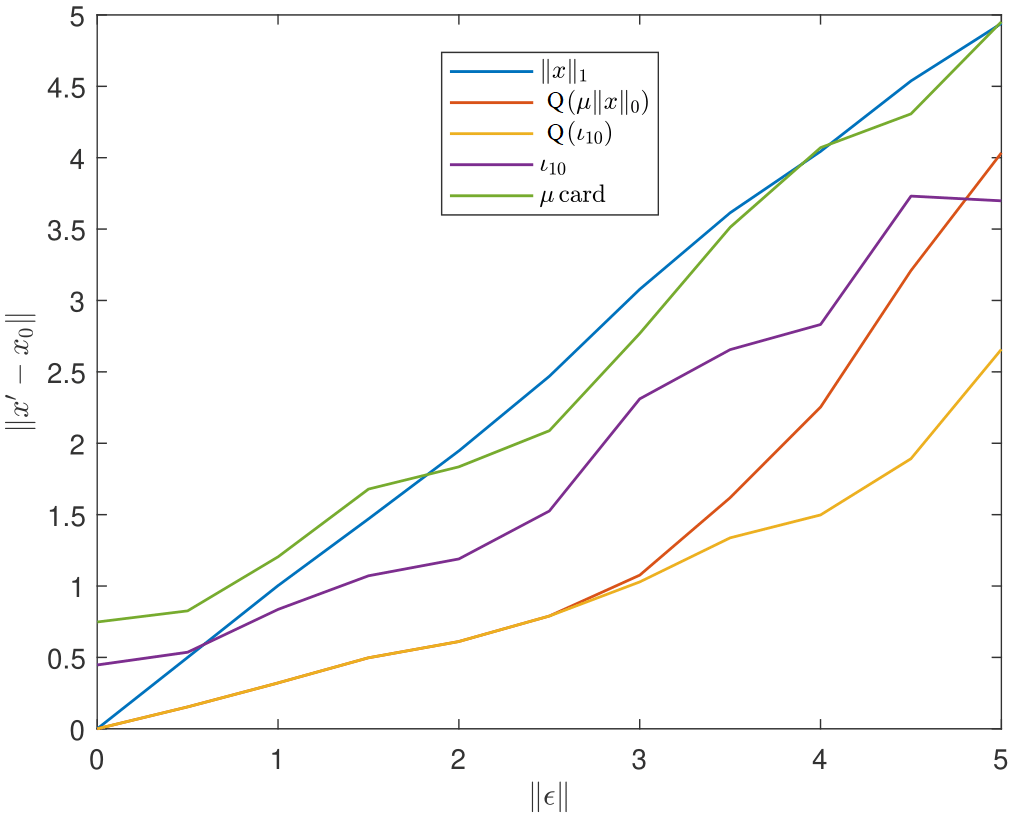}
     \includegraphics[width=.5\textwidth]{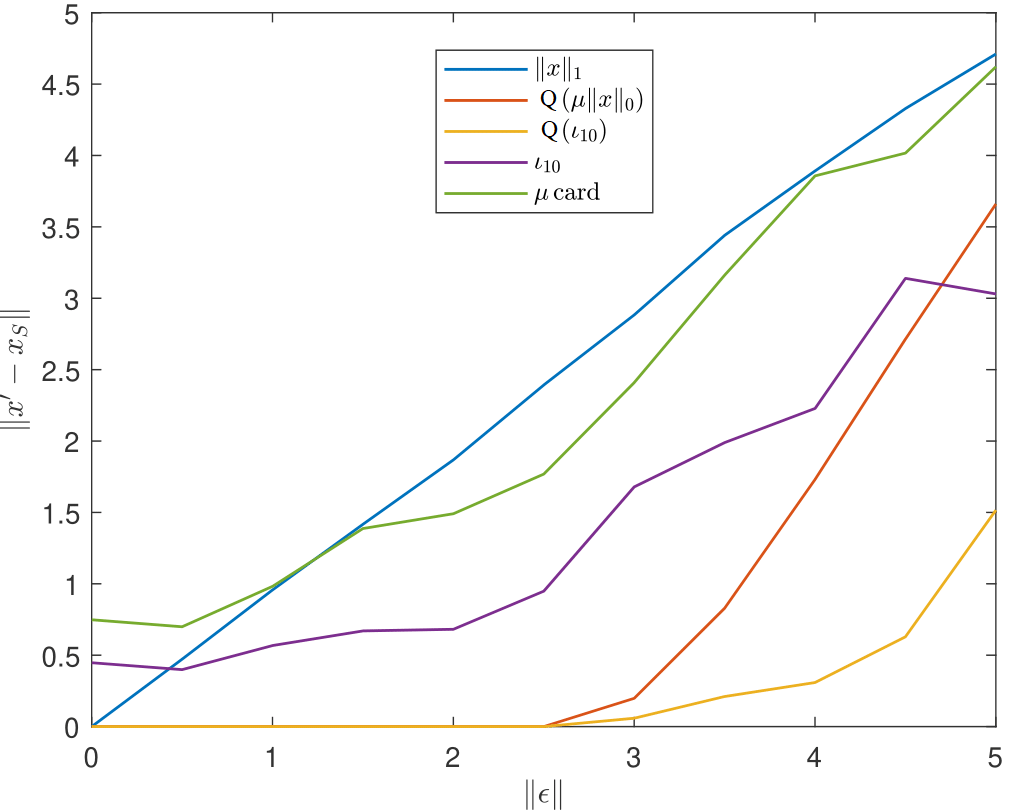}\\
 \caption{$\|x'-x_0\|_2$ (left) and $\|x'-x_S\|_2$ (right) versus $\|\epsilon\|_2$ for the 5 methods \eqref{l1probdual}, \eqref{q1}-\eqref{q1reg} and \eqref{q2}-\eqref{q2reg} minimized using with FBS. The methods based on $\Q_2(\card)$ and $\Q_2(\iota_{P_k})$ work perfectly down to $SNR\approx 4$.}\label{f4}
\end{figure}

Here we will present numerical results for the case of $k=10$, $m=100$ and $n=200$.
We use a matrix $A$ with Gaussian randomly generated columns, which are subsequently normalized, and solve problems \eqref{l1probdual}, \eqref{q1reg} and \eqref{q2reg} for $b=Ax_0+\epsilon$ for different levels of noise $\|\epsilon\|_2$ between 0 and 5. The vector $x_0$ has random entries between 2 and 4 in magnitude, and a total magnitude $\|x_0\|_2=11$. To solve the optimization problems we use FBS which is known to converge to a stationary point (by \cite{attouch2013convergence} in combination with Section 2.4 of \cite{carlsson2016convexification} or Section 6 of \cite{carlsson2018convex}).

We compare with $\ell^1$-minimization \eqref{l1probdual} as well as two forms of Iterative Hard Thresholding, which arise when applying FBS to the unregularized problems \eqref{q1} and \eqref{q2}. In the first case the proximal operator will simply threshold at $\sqrt{\mu}$ and in the latter threshold by keeping only the $k$ largest entries. Convergence of such algorithms are proven e.g.~in Section 5 of \cite{attouch2013convergence}, and convergence of the latter has also been shown in \cite{blumensath2009iterative} when $\delta_{3k}<0.18$. For this reason, we also included graphs for the result of minimizing \eqref{q1} and \eqref{q2} (labeled $\mu\card$ and $\iota_{10}$ in the plots). Each point on the respective curves is an average over 50 trials, where we have used 1000 iterations and with a step-size parameter of \(0.9/\|A\|^2\), which is close to the upper theoretical bound given in \cite{attouch2013convergence} (which coincides with the bound for the convex case, see e.g. \cite{combettes2005signal}).

To set the parameter $\lambda$ for the $\ell^1$-problem \eqref{l1probdual} we used the formula $$\lambda= \frac{\|\epsilon\|_2 }{\sqrt{n}}\sqrt{2 \log (n)}  $$ corresponding to the recommendations in Section 5.2 of \cite{chen2001atomic}. For \eqref{q1}-\eqref{q1reg} we used $\mu=1$ and $k$ was set to 10 for \eqref{q2}-\eqref{q2reg}, which we motivate as follows:

If the value of $\delta_{2k}^-$ is near 0, then the conditions in Corollary \ref{cor:doctorgadget} hold given that $2\sqrt{\mu}\lesssim\min \{|x_{0,j}|:|x_{0,j}|\neq 0\}$ where the latter in our case is $2.05$ and $\|\epsilon\|_2\leq \sqrt{\mu}$, whereas the conditions  in Corollary \ref{cor:doctorgadget2} hold as long as $3\|\epsilon\|_2\lesssim 2.05$. In both cases, the estimate for $\|x'-x_0\|_2$ reads $\|x'-x_0\|_2\lesssim \|\epsilon\|_2$ which is supposed to hold at least for $\|\epsilon\|_2\lesssim 2/3$. Despite the fact that $\delta_{2k}^-\approx 0$ is quite unlikely (as we saw in Section \ref{sec:size}), the graph in Figure \ref{f4} (left) indicates that the reality looks even better. Both algorithms find the oracle solution in 100\% of the trial for $\|\ep\|_2$ up to 2.5, and the true bound (for this particular example) seems to be $\|x'-x_0\|_2\lesssim \frac{1}{3}\|\epsilon\|_2$ for both \eqref{q1reg} and \eqref{q2reg}, whereas the true constant for $\ell^1$ is around 1 (despite $C_{10}=\infty$ as seen in Section \ref{sec:size}, as $\delta_{20}$ with high likelihood is greater than 0.4 \cite{blanchard-cartis-tanner-2011}).

 The first version of IHT, i.e.~minimization of the unregularized functional \eqref{q1}, is similar to $\ell^1$ in performance, whereas \eqref{q2} is slightly better. Comparing with \eqref{q1reg} and \eqref{q2reg} the benefits of using the quadratic envelope are undeniable. Note that all 3 methods work for noise-levels much greater than stipulated by the theory. We also remark that, rather surprisingly, there is no major difference between \eqref{q1reg} and \eqref{q2reg} for moderate noise levels. However, both these methods are designed to find the oracle solution $x_{or}$, not $x_0$, so to evaluate this performance we include in Figure \ref{f4} (right) also the graph of $\|x'-x_{or}\|_2$ versus $\|\epsilon\|_2$. From this we deduce that both work perfectly until $\|\epsilon\|_2=2.5$, but that \eqref{q1reg} deteriorates substantially faster beyond this point. In other words, in this example both methods based on $\Q_2(\mu \card)$ and $\Q_2(\iota_{P_{10}})$ work as expected down to $SNR$ around 4. In \cite{carlsson2020perfect} a much more thorough comparison between \eqref{l1probdual}, the two methods considered here, and other popular techniques such as Reweighted $\ell^1$ \cite{candes2008enhancing} and Huber-fitting \cite{selesnick2017sparse} is carried out. This paper also optimizes over hyperparameters, as opposed to fixing them a priori as in this section. %In short, the conclusion is that the methods based on the quadratic envelope are superior in all three domains: accuracy, robustness to noise, evaluation speed. The closest competitor is reweighted $\ell^1$ which has similar performance except that it is slower.
We refrain from similar experiments here since this is a theoretical paper and the above experiment was designed to illustrate and validate the theory, not to compare optimal performance of algorithms.

 %Note that, in the best case scenario $\sqrt{1-\delta_k^-}=\sqrt{1-\delta_{2k}^-}=1$, and then a simple computation shows that Corollary \ref{cor:doctorgadget} and \ref{cor:doctorgadget2} then applies for $SNR$ down to $2\sqrt{10}\approx 6$, and hence there is almost perfect harmony between theory and numerical results. More precisely, we can allow 50\% more noise in practice than predicted by the theory.

%\begin{wrapfigure}{r}{0.5\textwidth}
%  \begin{center}
%     \includegraphics[width=.5\textwidth]{x0_50x200.pdf}\\
%  \end{center}
% \vspace{-0.5cm}\caption{Same as Figure \ref{f5} but with 50 only rows in $A$.}\label{f6}\vspace{-0.5cm}
%\end{wrapfigure}
%In Figure \ref{f6} we show the same graphs except that now $A$ has size $50\times 200$. Clearly this has a significant impact on performance. In particular, although \eqref{q1reg} and \eqref{q2reg} still do better than traditional $\ell^1$-minimization, there is no longer a significant difference. This could indicate that the convex $\ell^1$-method is more reliable in very difficult scenarios, as opposed to the non-convex methods suggested here, but this would have to be further investigated to be confirmed. A drawback of $\ell^1$-methods is that one often needs to find a suitable $\lambda$, which leads to slow evaluation in practice.

Another issue that we have not discussed is the starting point. We have used $0$ for all examples above, and (a bit surprisingly) this seems to work better than using the least squares solution \(x_{LS}\) of $Ax=b$, which seems to have many local minima near it when we use $\Q_2(\mu\card)$. This is clearly seen in our final graph  \begin{wrapfigure}{r}{0.5\textwidth}
  \begin{center}
     \includegraphics[width=.5\textwidth]{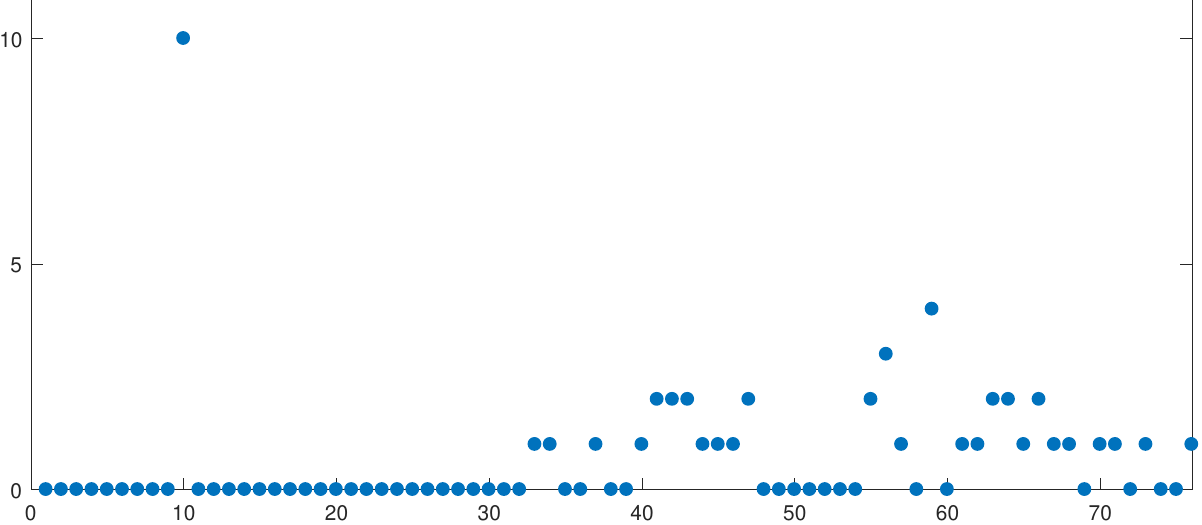}\\
  \end{center}
\vspace{-0.5cm} \caption{Histogram of cardinality for 50 trials of \eqref{q1reg} with $\|\epsilon\|_2=2.5$.}\label{f5}
\end{wrapfigure} where we plot a histogram of the cardinality of $x'$ over 50 trials with the noise level $\|\epsilon\|_2=2.5$, using $\Q_2(\card)$ and \(x_{LS}\) as starting point. Concerning $\Q_2(\card)$ it is interesting to note the following dichotomy, either the cardinality is around 10, or substantially larger, as predicted by Theorem \ref{thm:statpoint:vec}. For this noise level and starting point $x_{LS}$, $\Q_2(\iota_{10})$ still works perfectly, which is why its performance is excluded; the histogram hits 50 at $k=10$, in accordance with Corollary \ref{new}. Combined with Figure \ref{f4}, this underlines that when $k$ is known, $\Q_2(\iota_{10})$ is the best penalty.

\subsection{Implementation technicalities}\label{int tch}
Basically anywhere there is a method involving a sparsity inducing $\|x\|_1$-term, it can be easily replaced with $\Q_\gamma(\mu\card)$ or $\Q_{\gamma}(\iota_{P_k})$ if the model order is known. We encourage the reader to try these on his or her particular problem, and to facilitate this we here discuss briefly some implementational aspects and parameter choices. Code for evaluation of the corresponding proximal operators is available at the following GitHub repository: \begin{center} \texttt{https://github.com/Marcus-Carlsson/Quadratic-Envelopes} \end{center}   

First of all we note that it is often customary to put a factor $1/2$ in front of the quadratic term in \eqref{genprobS} and moreover the quadratic envelope depends on a parameter $\gamma$ which we have throughout kept fixed at 2. A more general version of \eqref{genprobS} would be
\begin{equation}\label{genprobSupp}
\Q_\gamma(f)(x)+\frac{1}{2}\|Ax-b\|_2^2, \quad \gamma>0.\end{equation}
To pass between various normalizations, we note that given any $\alpha>0$ one has $$\alpha \Q_{\gamma}(f)=\Q_{\alpha\gamma}(\alpha f),$$
so in particular \eqref{q1reg} is equivalent with $\Q_1(\frac{\mu}{2}\card)+\frac{1}{2}\|Ax-b\|_2^2$ and \eqref{q2reg} with $\Q_1(\iota_{P_k})+\frac{1}{2}\|Ax-b\|_2^2$, and the entire paper could as well have been written in this setting.

In order for the global minima of $f(x)+\frac{1}{2}\|Ax-b\|_2^2$ to not move when switching to \eqref{genprobSupp}, the general theory of \cite{carlsson2018convex} states that $\gamma$ should be less than $\|A\|^2$. In practice, this is too conservative. Reformulated in the general context \eqref{genprobSupp}, the condition $\|A\|_{\infty,col}\leq 1$
turns into $$\|A\|_{\infty,col}\leq \sqrt{\gamma},$$ so by this we should set $\gamma=\sqrt{\|A\|_{\infty,col}^2}$ in general. This is a much more realistic estimate in practice, but still it is given by a theoretical upper bound. We recall that $\gamma$ equals the maximum negative curvature of $\Q_\gamma(f)(x)$, and hence lowering the value of $\gamma$ makes the penalty ``less non-convex'', intuitively speaking. We have found that, for the problems considered in this paper, values of $\gamma$ as low as $0.3$ give better performance (i.e.~less chance of getting stuck in local minima), while still maintaining the property of finding the oracle solution. With that said, optimal parameter choices will be investigated elsewhere.

Concerning algorithms to minimize \eqref{genprobSupp}, we have found no significant difference between ADMM and FBS. The latter is guaranteed to converge to a stationary point when applied to \eqref{genprobS} (under mild assumptions). This follows by the main result of \cite{attouch2013convergence} combined with Section 6 of \cite{carlsson2018convex}. ADMM on the other hand, to our best knowledge, still lacks a proof of convergence at least for the non-separable penalty $\Q_{\gamma}(\iota_{P_k})$, (the separable case, which does apply to \eqref{q1reg}, is considered in \cite{wang2019global}). %Finally, the proximal operators of both $\Q_\gamma({\mu}\card)$  and $\Q_{\gamma}(\iota_{P_k})$ are available at ....

\iffalse
In fact, $\Q_2(\card)$ and $\Q_2(\iota_{P_k})$ can be seen as extreme points of a more general framework proposed in \cite{larsson-olsson-ijcv-2016}. Although that paper is written in the setting of matrices and without using the $\S_2$-transform explicitly, it is shown that $S_2^2(f)$ is computable for any functional $f$ of the form $$x\mapsto g(\card(x))$$ where $g$ is a convex increasing function on $\mathbb{N}$. Thus $\card(x)$ arise from the concrete choice $g(t)=t$ whereas $\iota_{P_k}(x)$ appears as $g(t)=\iota_{\{t:~t\geq K\}}$. The former is reasonable to use when we have maximum uncertainty about the sought cardinality, the latter when we have no uncertainty. Clearly the more common situation is that of limited uncertainty of the sought support, and the framework clearly allow to tailor-make the function $g$ to fit prior knowledge. 
\fi

\section{Conclusions}
With the wealth of papers analyzing sparsity-inducing penalties, is there a need for yet another one? The existing literature can be divided into two groups, either the results are asymptotic in nature (hence say little in a concrete setting) or they assume that $\delta_{2k}$ (or some analogous quantity) is sufficiently small. As argued in Section \ref{sec:size}, for the case $m=2n$, this forces the sparsity $k$ to be well below than 1\% of $n$ to achieve $\delta_{2k}\approx 0.4$ or less. On the other hand, in many concrete applications $k$ is substantially larger, and so there is a vast regime where there is \textit{no theoretical support} for that either $\ell^1$-minimization \eqref{l1probdual} or IHT gets anywhere near the ground truth.

The majority of our results, on the other hand, applies as long as any $2k$ columns of $A$ are linearly independent, for then $\delta_{2k}^-<1$, with the natural catch that if $\delta_{2k}^-$ is poor then a large SNR is needed. This is a significant theoretical improvement; if we are in the range $k/n>0.01$, then $\ell^1$-minimization \eqref{l1probdual} is
\textit{convex} and therefore e.g.~FBS applied to it is guaranteed to converge to some point $x_{1}'$, but by the results of \cite{attouch2013convergence} and \cite{carlsson2018convex}, the same is true for \eqref{q1reg} and \eqref{q2reg}, it just may happen that the convergence point $x_2'$ is not the global minimum. However, whereas there is \textit{no support} for the hypothesis that $x_1'$ is anywhere near ground truth, the theorems of this paper states that \textit{if} $x_2'$ is the global minimum, then it is the oracle solution which is the best possible outcome (and else it may not be near ground truth, just like $x_1'$). This gives the two methods studied here a significant theoretical advantage over $\ell^1$-minimization, (or IHT or reweighted $\ell^1$ as well for that matter). Combined with the numerical section which demonstrates superior performance in the entire range, this paper challenges the $\ell^1$-penalty as the penalty of choice for compressed sensing and sparsity based methods in general.

Finally, this paper studies design of sparsity inducing functionals, \textit{not algorithms to find their global minima or stationary points.} We prove that the global minima, under verifiable conditions, is the oracle solution. The fact that both ADMM and FBS (with 0 as starting point) seems to converge to the global minima is a numerical observation whose proof we leave as an open question.

\bibliographystyle{plain}
\bibliography{bib_IP}

\section{Appendix}
\subsection{Appendix to Section \ref{sec:unique}} \label{app:prel}

While it is possible to deal with gradients and subdifferentials in $\C^n$ by simply identifying it with $\R^{2n}$ in the canonical way, the calculus becomes more intuitive if avoid this step. Instead, we say that a function $g_d:\C^n\rightarrow \R$ is differentiable at a point $x$ is there is a vector $v\in\C^n$ such that \begin{equation}\label{defgrad}\lim_{\|y\|\rightarrow 0^+}\frac{g_d(x+y)-g_d(x)-\re\scal{y,v}}{\|y\|}=0.\end{equation} In this case we write $v=\nabla g_d(x)$.
For example, consider the function $g_d(x)=\|Ax-b\|^2$. Upon noting that $g_d(x+y)=\|Ax-b\|^2+2\re\scal{y,A^*(Ax-b)}+\|Ay\|^2$, it readily follows that $\nabla g_d(x)=2A^*(Ax-b)$.
Similarly, if $g_c$ is convex and $v$ is a vector such that $$g_c(x+y)-g_c(x)-\re\scal{y,v}\geq 0$$ for all $y$, we say that $v$ is in the subdifferential of $g_c$ which we denote by $v\in\partial g_c(x)$.

Let us establish the claim following \eqref{defstationary}, i.e.~that a function $g$ of the type $g_c+g_d$ for functions as above has a stationary point at $x$ if and only if $-\nabla g_d(x)\in\partial g_c(x)$. The condition \eqref{defstationary} for stationarity translates to
$$0\leq \underset{y\neq 0}{\liminf_{y\rightarrow x}}~ \frac{g(x+y)-g(x)}{\|y\|}=\underset{y\neq 0}{\liminf_{y\rightarrow x}}~ \frac{g_c(x+y)-g_c(x)+\re\scal{y,\nabla g_d(x)}}{\|y\|}.$$
To see this, just add and subtract $\re\scal{y,\nabla g_d(x)}$ to the numerator and invoke \eqref{defgrad}. It immediately follows that if $-\nabla g_d(x)\in\partial g_c(x)$ holds then $x$ is stationary. Conversely, suppose that $x$ is stationary. If there exists a $y$ such that $g_c(x+y)-g_c(x)+\re\scal{y,\nabla g_d(x)}<0$, then for $t\in[0,1]$ we have by convexity that $g_c(x+ty)\leq tg_c(x+y)+(1-t)g_c(x)$ so $$g_c(x+ty)-g_c(x)+\re\scal{ty,\nabla g_d(x)}\leq t(g_c(x+y) -g_c(x)+\re\scal{y,\nabla g_d(x)})$$
by which it follows that the above $\liminf$ must also be $<0$, a contradiction.

\subsection{Appendix to Section \ref{s4}}

The full statement of Theorem \ref{celok5} follows by combining the below four lemmas. For concreteness assume that we work over $\C^n$.

\begin{lemma}\label{l4}
Without any restriction on $A$, the functional $\kfi$ attains its infimum.
\end{lemma}
%\begin{proof}
%We prove this by induction over $n$, where $A$ is $m\times n$ and $m$ is fixed. For $n=1$ the statement is easy to see. Assume $n>1$ and that the result is true for matrices of size  $m\times (n-1)$. There are $n$ linear subspaces obtained by setting one of the coordinates to zero, and $\kfi$ restricted to each of these subspaces attains its infimum over the respective subspace, by the inductive hypothesis. If $H$ denotes the union of such subspaces we thus have that the infimum of $\kfi$ over $H$ is attained. Outside of $H$ we have $\kfi(x)=n+\|Ax-b\|^2$. The right hand side is minimized on the affine subspace defined by the equation system $A^*Ax=A^*b$. If this affine subspace contains no point in $H$, then clearly the infimum in $\C^n\setminus H$ is attained and hence it is attained globally, which was to be shown. If the affine subspace and do $H$ intersect, then clearly $\kfi$ assumes lower values in $H$ than the infimum over $\C^n\setminus H$, and hance the global infimum must be in $H$, in which case it is also attained.
%\end{proof}
\begin{proof} Fix \( 1 \le j \le n  \) and consider submatrices \( A(:,J) \) of \(A\) with \(m\) rows and \(j\) columns, \(J \subseteq \{ 1, \dots , n \} \) and \( \#J=j \); \(J\) determines which columns of \(A\) are selected. Now for each fixed \( J \), the minimum of $ \| A(:,J) x -b \|^2 _2 $ is attained and can be computed by solving the normal equations. Let $c_j$ be a corresponding vector in $\C^n$ with zeroes off $J$, such that  $\| A c_{j} -b \|^2 _2   $ equals the minimum in question. Among the $\{ c_{J}\}_{\#J=j}$ we denote by $c_j$ one that satisfies \[ \|Ac_j - b \|^2 _2 = \min_{ \#J=j} \min_{x \in \mathbb{C}^j} \| A(:,J) x - b \|^2 _2.    \] If \( I= \inf \mathcal{K}_\mu (x)  \) we can select a sequence \( x_i \in \mathbb{C}^n  \) such that \( \mathcal{K}_\mu (x_i) \to I  \). By construction it must be \[ \mathcal{K}_\mu ( c_{\text{card}(x_i)}  ) \le \mathcal{K}_\mu (x_i). \] Since \( \mathcal{K}_\mu (x_i) \) is arbitrarily close to \(I\) and the \(c_j\) are finite, it must exist a \( \bar{j}  \) - at least one - such that \( \mathcal{K}_\mu (c_{\bar{j}})=I. \)  \end{proof}

\begin{lemma}\label{l5}
If $\|A\|_{\infty,col}\leq 1$, the functional $\kfir$ attains its infimum, which equals that of $\kfi$.
\end{lemma}
\begin{proof}
In the light of the basic inequality $\kfir\leq \kfi$ and the previous lemma, the two infima can only be different if there exists a point $x_0$ such that $\kfir(x_0)<\inf \kfi$. We prove by contradiction that this is impossible. In particular $\kfir(x_0)<\kfi(x_0)$, which implies that $\Q_2(\mu\card)(x_0)<\mu\card(x_0)$ since the quadratic terms are the same. This in turn implies (by \eqref{l0000}) that there must be some index $j$ such that the corresponding value in $\Q_2(\mu\card)(x_0)$ is different from $\mu\card (x_{0,j})$, which happens if and only if \begin{equation}\label{py1}0<|x_{0,j}|<\sqrt\mu.\end{equation} Let $e_j$ equal $1$ in coordinate $j$ and zero elsewhere and consider $$t\mapsto \kfir(x_0+t \frac{x_{0,j}}{|x_{0,j}|}e_j)$$ for real $t$ such that $0<|x_{0,j}|+t<\sqrt\mu$. This must be a quadratic polynomial, again by inspection of \eqref{l0000}, which also gives that
\begin{equation}\label{py}\frac{d^2}{dt^2}\kfir(x_0+ t \frac{x_{0,j}}{|x_{0,j}|} e_j)\Big|_{t=0}=-2+2\|a_j\|_2^2\leq 0.\end{equation} Hence this quadratic polynomial attains its minimum over the stated range at an endpoint.

It follows that we can redefine $x_{0,j}$ to equal either $0$ or $\sqrt{\mu}$, so that the resulting point $x_1$ satisfies $\kfir(x_1)\leq \kfi(x_0)$. We can now continue like this for another index $j$ such that \eqref{py1} holds (if it exists), and this process must terminate after finitely many steps $N$. Denoting the resulting point by $x_N$, we see that it satisfies $\kfi(x_N)=\kfir(x_N)<\inf \kfi$, a contradiction. Hence $\inf\kfi=\inf\kfir$.

Let $x_0$ be a point where the first infimum is attained. Then $\kfir(x_0)\leq \kfi(x_0)$ so we must have identity and hence the infimum of $\kfir$ is also attained.

\end{proof}

\begin{lemma}\label{l6}
Let $\|A\|_{\infty,col}\leq 1$ and let $x_0$ be a global minima of $\kfir$ which is not a global minima for $\kfi$. Then it belongs to a connected set of global minima of $\kfir$ including at least two global minima of $\kfi$.
\end{lemma}
\begin{proof}
By repetition of the previous proof we conclude that the first and second derivative of $\kfir(x_0+te_j)$ must be equal to 0, so the quadratic polynomial is constant in the range $0<|x_{0,j}|+t<\sqrt\mu$. Setting $t$ to be one fo the endpoints gives two new global minimizers $x_1$ with either $x_{1,j}=0$ or $|x_{1,j}|=\sqrt{\mu}$. Either $x_1$ is a minimizer of $\kfi$ or we can continue the process with another subindex. The result now easily follows.
\end{proof}

\begin{lemma}\label{l7}
Let $\|A\|_{\infty,col}< 1$, then any local minima of $\kfir$ is a local minima of $\kfi$. In particular, the sets of global minimizers coincide.
\end{lemma}
\begin{proof}
Let $x_0$ be a local minimizer of $\kfir$ but not of $\kfi$. We again repeat the arguments in Lemma \ref{l5}, but this time we get strict inequality in \eqref{py}, which is impossible. Hence such minimizers do not exist.

If now $x_0$ is a global minimizer to $\kfir$ then it is a local minimizer of $\kfi$, which in the light of $\kfi\geq\kfir$ means that it is a global minimizer, and the proof is complete.
\end{proof}

\subsection{Appendix to Section \ref{Kfeas}}\label{kfesappendix}

The proof will follow after a collection of minor results.

\begin{proposition}\label{propnegscal}
For any $m+2$ vectors $v_1,\ldots,v_{m+2}$ in $\R^m$, we can always pick two such that $\scal{v_i,v_j}\geq 0$.
\end{proposition}
We note that the proposition is sharp since the $m+1$ vortices of a simplex in $\R^m$ do have negative scalar products.
\begin{proof}
This follows from a simple induction argument. It is indeed clear in \( \mathbb{R} \). Suppose now we take \( m+2 \) vectors \( v_i\) such that \( \langle v_i , v_j \rangle < 0  \) if \( i \ne j  \). If \(V\) is the hyperplane perpendicular to \( v_{m+2}  \), the projections \( v_1', \dots , v_{m+1}' \) of \( v_1, \dots , v_{m+1} \) on \( V  \) must also have negative scalar products (since the projections onto $v_{m+2}$ always are in the same direction, opposite that of $v_{m+2}$, so $\scal{v_i',v_j'}<\scal{v_i,v_j}$). Since \( V\) is an \((m-1)\)-dimensional vector space, the desired result is immediate by induction.
\end{proof}

Recall that $a_1,\ldots, a_n$ denote the columns of $A$.
\begin{lemma}\label{mfisR}
Let $T\subset\{1,\ldots,n\}$ have cardinality $\#T\geq n-k$ and consider $\{a_j\}_{j\in T}.$ Under assumption (A1-A2), we can pick indices $i,j\in T$ such that $\|a_i-a_j\|^2< 2$. \end{lemma}

\begin{proof}
We first consider the real case $\R^m$. Then $\#T\geq m+2$ by (A1) and since $\|a_i-a_j\|^2=\|a_i\|^2-2\re\scal{a_i,a_j}+\|a_j\|^2$, the result is immediate by (A2) and Proposition \ref{propnegscal}. Finally, since $\C^m$ is isomorphic with $\R^{2m}$, the corresponding result in the complex case follows analogously, since now $\#T\geq 2m+2$ by (A1).
\end{proof}
Armed with the above statements we can now start to characterize global minimizers of $\ksir$, which is annoyingly difficult. It is even difficult to prove that they exist, so as a first step we shall restrict attention to a closed ball. Recall that $\D$ denotes either the unit disc in $\C$ or, if we work over the reals, the interval $[-1,1]$.
\begin{lemma}\label{ko0}
There exists an $R_0>0$ such that for any $R>R_0$, any global minimum $x'$ of $\ksir$ restricted to $(R\D)^n$ must satisfy $$|\tilde x'_{k+1}|\leq\frac{R}{2}.$$
\end{lemma}
\begin{proof}
Introduce $$U=\left\{x\neq 0:~|\tilde x_{k+1}|\geq\frac{1}{2}|\tilde x_1|\right\}.$$ We first note that $\Q_2(\iota_{P_k})(x)>0$ for all $x\in P_k^c$, which follows by the definition (see \eqref{Qdef}), so in particular this holds for all $x\in U$. Define \begin{equation}\label{olbia}\alpha =\inf\left\{\Q_2(\iota_{P_k})(x): ~x\in U,~\|x\|_2=1\right\}.\end{equation}
Since we are minimizing a continuous (non-zero) positive function over a compact set, $\alpha>0$. Let us write $s=s_x$ for the function defined in \eqref{lust}, when there is a need to make the dependence on $x$ clear. The function $s$ is radially dependent, i.e.~$s_{tx}=t s_x$ for $t\in\R$, and hence $k_*$ is radially independent. Looking at the expression for $\Q_2(\iota_{P_k})$ we see that $$\Q_2(\iota_{P_k})(tx)=t^2\Q_2(\iota_{P_k})(x)\quad t\in \R.$$
Note that $\ksi(0)=\ksir(0)=\|b\|_2^2$ so the global minimum of $\ksir$ is less than or equal to this. Let $R_0$ be such that $\alpha (R_0/2)^2>\|b\|_2^2$. If $x\in U$ satisfies  $\|x\|_2>(R_0/2)$, then $$\ksir (x)\geq \Q_2(\iota_{P_k})(x)\geq \alpha \|x\|^2_2>\|b\|_2^2$$ so it follows that such a point is no global minimizer (at least not on any set containing 0).

Now let $R$ and $x'$ be as stated in the lemma. If $|\tilde x_{k+1}'|> R/2$ then clearly $\|x'\|_2\geq R_0/2$ so $x'$ can not be in $U$. But then $ \frac{1}{2}|\tilde x_1'|>|\tilde x_{k+1}'| > R/2$ which means that $|\tilde x_1'|$ is outside of $R\D$. This is impossible, so the proof is complete.
\end{proof}

We define an angle of a complex number $z$ to be any number $\alpha_z$ such that $z=|z|e^{i\alpha_z}$. While this is unique modulo $2\pi$ for $z\neq 0$, it can be any number for $z=0$. Recall that $e_1,\ldots,e_n$ denotes the canonical basis in $\R^n$ (or $\C^n$).

\begin{lemma}\label{ko1}
Let $x$ be any vector and let $p,q\in \{1,\ldots,n\}$ be different indices such that $|x_p|\leq |\tilde x_{k+1}|$ and $0<|x_q|\leq |\tilde x_{k+1}|$ holds. Fix corresponding angles $\alpha_p$ and $\alpha_q$ and set
$$x(t)=x+te^{i \alpha_q} e_p-t e^{i\alpha_q}e_q.$$
Then $\Q_2(\iota_{P_k})(x(t))$ is twice differentiable at 0 and $$\frac{d^2}{dt^2}\Q_2(\iota_{P_k})(x(t))\Big|_{t=0}=-4.$$
\end{lemma}
\begin{proof}
This is relatively easy to see in the case when $|x_p|$ and $|x_q|$ are strictly less than $|\tilde x_k|$, so we first assume this. Then the two points where $t$ show up in the sequence $\tilde x(t)$ are beyond $k$, assuming $t$ is kept small enough. For any $l\geq 1$ we then have that \begin{equation}\label{india}\sum_{j>k-l}|\tilde x_j(t)|=\sum_{j>k-l}|\tilde x_j|,\end{equation} because the left hand side includes one term like $|\tilde x_p| +t$ and one term like $|\tilde x_q|-t$, which therefore cancel out. (This is were we used $|x_q|>0$). Looking at the expression \eqref{lust} which is used to determine $k_*$, we see that all the values $s_{x(t)}(l)$ are unaffected by small $t$, and hence $k_*$ is unaffected by $t$ (as long as it is small enough). Now, the first part of the expression \eqref{lust1} for $\Q_2(\iota_{P_k})(x(t))$ also contain \eqref{india} (for the particular value $l=k_*$), and hence this is constant. The second part equals $$-\sum_{j>k-k_*}|\tilde x_j(t)|^2=-\sum_{j>k-k_*}|\tilde x_j|^2-2|x_p|t+2|x_q|t-2t^2,$$ whose second derivative at 0 equals -4, as was to be shown.

Now assume that $|x_p|$ or $|x_q|$ (or both) equals $|\tilde x_k|$. The conclusion will follow as above, once we verify that $i)$ $k_*$ is invariant for small $t$ and, $ii)$ both terms with $t$ in them appear in $\{|\tilde x_j(t)|\}_{j>k-k_*}$.

To see $i)$, let $a$ be the largest integer such that $|\tilde x_{k+1-a}|=|\tilde x_k|$ and note that $s_x(1)>0$ since we have assumed $|\tilde x_{k+1}|>0$. Moreover, by inspection of \eqref{lust} we have that $s_x(l)=s_x(1)$ for all $1\leq l \leq a$, so $k_*\geq a$. By this it follows, if we write $k_*(t)$ for the $k_*$ associated with $x(t)$, that we also have $k_*(t)\geq a$ for small $t$, by continuity. Moreover both terms with $t$'s show up in $\{|\tilde x_j(t)|\}_{j>k-l}$ for all $l\geq a$, so for such $l$ we have that $s_{x(t)}(l)$ is unaffected by small $t$'s by the same cancelation effects as in \eqref{india}. By this we finally conclude that $k_*(t)$ is constant in a neighborhood of 0, i.e.~$i)$. Since we also know $k_*\geq a$, $ii)$ follows as well by what was written above. The proof is complete.
\end{proof}

\begin{proof}[Proof of Theorem \ref{celok6}]
Let $x'$ be a local minimizer of $\ksir$, and assume that $x'\not\in P_k$. We first assume that all values $x'_j$ are non-zero, and let $\alpha_j$ be corresponding angles. The set $T=\{j:~|x_j|\leq |\tilde x_{k+1}|\}$ clearly satisfies $\#T\geq n-k$, so we can use Lemma \ref{mfisR} on the matrix with columns $\{e^{i\alpha_j}a_j\}_{j=1}^n$ to pick two indices $p$ and $q$ such that \begin{equation}\label{gp}\fro{e^{i\alpha_p}a_p-e^{i\alpha_q}a_q}<2.\end{equation}
By the choice of $T$, we also have that Lemma \ref{ko1} applies. Let $x(t)$ be as in that lemma. It then follows that $\frac{d^2}{dt^2}\ksir(x(t))$ exists at 0 and equals
$$-4+2\fro{e^{i\alpha_p}a_p-e^{i\alpha_q}a_q}<0.$$
This contradicts the assumption that $x$ is outside $P_k$, which hence must be false.

We still need to consider the case when some values $x'_j$ are 0. In this case we pick $x_q$ as in Lemma \ref{ko1} and we let $p$ be any index such that $x_p=0$. The angle $\alpha_p$ can now be chosen such that \eqref{gp} holds, which leads to a contradiction as before.

It is now established that all local minimizers of $\ksir$ lie in $P_k$, and clearly they are also local minimizers of $\ksi$ in view of $\ksi\geq\ksir$ and the fact that these two coincide on $P_k$. Next we turn to prove that they exist. Fix $R>R_0$ as in Lemma \ref{ko0} and let $x'$ be a global minimizer of $\ksir$ in $(R\D)^n$. By the lemma we have $|\tilde x'_{k+1}|<R/2$, so any perturbation $x(t)$ as considered in Lemma \ref{ko1} stays within $(R\D)^n$. With this at hand, we conclude as above that $x'\in P_k$.

However, on $P_k$ both $\ksir(x)$ and $\ksi(x)$ coincide with simply $\|Ax-b\|^2$, the minimum of which is attained by the proof of Lemma \ref{l4}. We conclude that $\ksir$ do attain its global minima, and that $\ksir$ and $\ksi$ share global minimizers.

\end{proof}

\end{document}